\documentclass[11pt]{elsarticle}

\usepackage[frenchb,english]{babel}
\usepackage{dsfont,newlfont}
\usepackage{slashbox}
\usepackage{amssymb,amsmath,amsfonts}
\usepackage{array,multirow,ulem,dsfont}
\usepackage{mathtools,stmaryrd,enumitem}
\usepackage{color}

\usepackage{amsthm}

\newtheorem{thm}{Theorem}[section]
\newtheorem{defn}[thm]{Definition}

\newtheorem{lem}[thm]{Lemma}
\newtheorem{prop}[thm]{Proposition}

\newtheorem{rmq}[thm]{Remark}

\newtheorem{exmp}{Example}
\renewenvironment{proof}{\textit{Proof.}}{\hfill$\Box$\linebreak}

\newcommand{\IN}{\mathds{N}}

\newcommand{\IR}{\mathds{R}}

\newcommand{\IP}{\mathds{P}}

\newcommand{\p}{\mathcal{I}_h^k}

\newcommand{\dis}[1]{\displaystyle{#1}}

\renewcommand{\geq}{\geqslant}
\renewcommand{\leq}{\leqslant}

\begin{document}
\begin{frontmatter}

%%  title
\title{\Large Discontinuous Galerkin method for blow-up solutions of nonlinear 1D wave equations}

 \author[label1]{A. Azaiez}
 \author[label2]{M. Benjemaa}
 \author[label3]{A. Jrajria}
 \author[label4]{H. Zaag}
 
 \address[label1]{Department of Mathematics, New York University Abu Dhabi, Saadiyat Island, P.O. Box 129188, Abu Dhabi, United Arab Emirates. $\langle$asma.azaiez@nyu.edu$\rangle$}
 \address[label2]{Faculty of Sciences of Sfax, Sfax University, Tunisia. $\langle$mondher.benjemaa@fss.usf.tn$\rangle$}
 \address[label3]{Faculty of Sciences of Sfax, Sfax University, Tunisia. $\langle$aida.jrajria@gmail.com$\rangle$}
 \address[label4]{University Sorbonne Paris Nord, LAGA-CNRS, F-93420, Villetaneuse, France. $\langle$Hatem.Zaag@univ-paris13.fr$\rangle$}

\begin{abstract} We develop and study a time-space discrete discontinuous Galerkin finite elements method to approximate the solution of one-dimensional nonlinear wave equations. We show that the numerical scheme is stable if a nonuniform time mesh is considered. We also investigate the blow-up phenomena and we prove that under weak convergence assumptions, the numerical blow-up time tends toward the theoretical one. The validity of our results is confirmed throughout several numerical examples and benchmarks.
\end{abstract}


\end{frontmatter}


\section{Introduction}
This paper is concerned with the development of a numerical method, based on discontinuous Galerkin (DG) formulation, in order to approximate the blow-up behaviors of smooth solutions of the semilinear wave equation in one space dimension $\Omega = (a,b)\subset \IR$ with periodic boundary conditions
\begin{align}\label{eqq1}
\left\{
\begin{array}{ll}
\partial_{tt} u - \,\partial_{xx} u = |u|^p, & \ \text{in}\ \Omega\times (0,\infty) \\
u(0) = u_0,\quad \partial_t u(0) = u_1, & \ \text{in}\ \bar{\Omega} \\
u(a,t) = u(b,t), & \ t\geq 0.
\end{array}
\right.
\end{align}
with $p>1$. The theoretical study of the semilinear wave equation in well developed. In \cite{CF85} and \cite{Caffarelli and Friedman}, Cafarelli and Friedman showed the existence of solutions of Cauchy problems for smooth initial data and gave a description of the blow-up set. In \cite{Gla73},  Glassey proved that under suitable assumptions on the initial data, the solution $u$ of \eqref{eqq1} blows up in the following sense : there exists $T_\infty<\infty$, called the blow-up time, such that the solution $u$ exists on $[0,T_\infty)$ and $$\|u(.,t)\|_{L^\infty(\Omega)}\longrightarrow \infty\quad \text{as}\quad t\longrightarrow T_\infty.$$
Recently, Merle and Zaag gave in a series of papers a classification of the blow-up behavior and an exhaustive description of the geometry of the blowup set \cite{MZ,MR2147056,MR2362418,MZ12}. More theoretical results can also be found in \cite{Az,Caffarelli and Friedman,John,livine,Glassey}.\\

From a numerical point of view, the approximation of solutions which blow up in finite time is more delicate. Indeed, one of the major difficulties when deriving numerical schemes is related to the standard stability criterion which imposes the boundedness of the numerical solution at any finite time. This is clearly in opposition with the sought blow-up behavior. In addition, the numerical solutions may remain bounded though the exact solutions do explode in finite time. These aspects have been observed when using a spectral method or even a finite differences (FD) method for the Constantin-Lax-Majda equation \cite{Cons85,Cho}. To overcome such a difficulty, Nakagawa \cite{Nakagawa} first introduced an adaptive time-stepping strategy to compute the blow-up FD solutions and the blow-up time for the 1D semilinear heat equation $\partial_t u - \partial_{xx} u = u^2$ in $(0,\,1)$ with homogeneous Dirichlet boundary conditions. To ensure the stability of his numerical scheme, he defined a local time stepping given by $$\Delta t^n  = \tau \min \left( 1,\, \frac{1}{\Vert u^n_h\Vert_2}\right)$$ where $\tau$ is a prescribed parameter. He showed that the numerical solution converges point-wise toward the exact solution. Moreover, by setting the numerical blow-up time $$ T(\tau,\Delta x) = \sum_{n=0}^{\infty} \Delta t^n,$$ he proved that $T(\tau,\Delta x)$ is finite and converges toward the theoretical blow-up time when $\Delta x$ goes to zero. Since then, many authors have improved Nakagawa's results and showed that the FD schemes with adaptively-defined time mesh give good approximation for the blow-up solution of the nonlinear heat equation \cite{Abia,Chen,Cho et al}. Other methods using different approaches, such as finite elements methods, semi-discretization and line methods, rescaling techniques, etc. for the numerical approximation of blow-up solutions of parabolic equations can also be found in \cite{Berg,Bra04,Cho16,Ngu17} and references therein.\\

For hyperbolic equations, Cho applied Nakagawa's ideas to the nonlinear wave equation with nonuniform time mesh \cite{Cho}. Recently, Sasaki and Saito \cite{Sasaki} reduced the nonlinear wave equation to a first order system and considered a FD scheme with a local time stepping. They succeeded in proving the convergence of their FD scheme and the numerical blow-up time. It is worth noticing that almost all the methods we found in the literature are essentially based on FD discretizations, and only few use variational (integral) formulations \cite{Groisman06,Guo15,Holm18}. To the best author's knowledge, there are no previous works dealing with discontinuous Galerkin (DG) approximations for nonlinear wave equation with blow-up solution. We propose in this paper to investigate such a DG methods to numerically solve the semilinear wave equation \eqref{eqq1} when blow-up phenomena occur.\\
 
The organization of this paper is as follows. In Section \ref{Sec_DG_method}, we present the DG methods and we derive a numerical scheme for the nonlinear wave equation. Section \ref{Sec_Stability} is devoted to the proof of the stability of the proposed numerical scheme. % In Section \ref{sec_blowup}, % we establish the main results of this paper which are the convergence of the numerical solution and the numerical blow up time respectively toward the exact solution and the exact blow-up time. 
In Section \ref{sec_blowup}, we prove that the numerical blow-up time converges toward the exact blow-up time under weak convergence assumptions. Finally, we provide several numerical examples that illustrate the validity of our proposed method in Section \ref{Sec_exmp}.
%
%%%%%%%%%%%%%%%%%%%%%%%%%%%%%%%%%%%%%%%%%%%%%%%%%%%%%%%%%%%%%%%%%%%%%%%%%%%%%%%%%%%%%%%%%%%%%
%
\section{Discontinuous Galerkin method}\label{Sec_DG_method}
In this section, we derive a discontinuous Galerkin scheme (DG) for the non linear wave equation \eqref{eqq1}. Formally, one may rewrite the D'Alembert operator as $\square = \left(\partial_t - \partial_x\right)\left(\partial_t + \partial_x\right)$. Based on such a decomposition, we split \eqref{eqq1} into a first order system as follows:
\begin{align}\label{eq2}
\left\{
\begin{array}{ll}
\partial_t u + \,\partial_x u = \phi, & \text{in}\, (a,b)\times (0,\infty) \\
\partial_t\phi - \partial_x\phi = \vert u\vert ^p, & \text{in}\, (a,b)\times (0,\infty) \\
u(x,0) = u_0(x), & x\in (a,b)\\
\phi(x,0) = \phi_0(x) , & x\in (a,b), \\
u(a,t) = u(b,t), & t\geq 0\\
\phi(a,t)=\phi(b,t), & t\geq 0.
\end{array}
\right.
\end{align}
with $\phi_0 = u_1 + u_0'$.
\begin{rmq}
One could also prefer the factorization $\square = \left(\partial_t + \partial_x\right)\left(\partial_t - \partial_x\right)$. However, such a doing has no significant impact on the DG scheme.
\end{rmq}
\subsection{Space discretization}
In order to introduce a variational approximation of the system \eqref{eq2}, we consider a partition for the spatial domain $[a,b] = \bigcup_{i=1}^I K_i$ consisting of cells $K_i = [x_{i-\frac{1}{2}}, x_{i+\frac{1}{2}}]$, $1\leq i\leq I$. % where $$a = x_{\frac{1}{2}}<x_{\frac{3}{2}}<\ldots<x_{I-\frac{1}{2}}<x_{I+\frac{1}{2}} = b.$$
The length of the cell $K_i$ is denoted $ h_i = x_{i+\frac{1}{2}} - x_{i-\frac{1}{2}}$. For simplicity, we shall assume that $h_i = h>0$ for all $i$. Next, we define the finite dimensional space $V_h^k$ consisting of all functions $v$ such that their restriction on a cell $K_i$ is a polynomial of degree at most $k$, i.e.
\begin{equation*}
V_h^k = \left\{ v\ / \ v|_{K_i}  \in \IP_{k}[K_i],\  i = 1,\dots,I\right\},
\end{equation*}
where $\IP_{k}[K_i]$ denotes the space of polynomials in $K_i$ of degree less than or equal to $k$. In the sequel, we will consider the Lagrange polynomials, denoted $\langle \varphi^i_j\rangle_{1\leq j \leq k+1}$, as a basis of $\IP_{k}[K_i]$. %, i.e. $$\IP_{k}[K_i] = \langle \varphi^i_j\rangle_{1\leq j \leq k+1}\ \ \text{with} \ \ \varphi^i_j(x) = \prod_{\substack{\ell=1 \\ \ell\neq j}}^{k+1} \frac{(x-x_\ell^i)}{(x_j^i-x_\ell^i)}\ \ \forall\ x\in K_i,$$ where the $(x_j^i)_{1\leq j \leq k+1}$ are nodes %\footnote{In general, these nodes are chosen equidistant. In such a case, $x_j^i = x_{i-\frac{1}{2}} + \frac{j-1}{k}\,h_i$. Other choices, such as the Chebychev nodes, can also be considered.}
% within the cell $K_i$. 
Notice that the functions of $V_h^k$ are allowed to be discontinuous across the elements interfaces. The solutions of the numerical method are denoted by $u_h$ and $\phi_h$ and both belong to $V_h^k$. We denote by $(u_h)^-_{i+\frac{1}{2}}$ and $(u_h)^+_{i+\frac{1}{2}}$ the left and right limits of $u_h$  at $x_{i+\frac{1}{2}}$, respectively.  Moreover, we denote $[u_h]_{i+\frac{1}{2}} = (u_h)^+_{i+\frac{1}{2}} - (u_h)^-_{i+\frac{1}{2}}$ the jump of $u_h$ at the cell interface $x_{i+\frac{1}{2}}$. The same notations apply also to $\phi_h$. Multiplying the system \eqref{eq2} by test functions and integrating over the cells yields the following variational formulation: find  $(u_h,\phi_h)\in V_h^k\times V_h^k$ such that for all test functions $(\varphi_h,\psi_h)\in V_h^k\times V_h^k$ and for any $1\leq i \leq I$
\begin{subequations}\label{semi_disct_scheme_for_1}
\begin{gather}
\begin{aligned}
\int_{K_{i}} \partial_t u_h\, \varphi_h dx & - \int_{K_{i}} u_h\,  \partial_x  \varphi_h dx \\
& + (\widehat{u}_h\,\varphi_h)_{i+\frac{1}{2}} - (\widehat{u}_h\,\varphi_h)_{i-\frac{1}{2}} = \int_{K_i}\phi_h\, \varphi_h dx \label{semi_disct_scheme_for_u}
\end{aligned} \\
\begin{aligned}
\int_{K_{i}}\partial_t \phi_h\, \psi_h dx & + \int_{{K_i}}\phi_h\, \partial_x\psi_h dx \\
& - (\widehat{\phi}_h\,\psi_h)_{i+\frac{1}{2}} + (\widehat{\phi}_h\,\psi_h)_{i-\frac{1}{2}} = \int_{K_i}\p\left(|u_h|^p\right) \psi_h dx,\label{semi_disct_scheme_for_phi}
\end{aligned}
\end{gather}
\end{subequations}
where $\widehat{u}_h$ and $\widehat{\phi}_h$ are the numerical fluxes and have to be defined at the cell interfaces, and $\p : C([a,b])\rightarrow V_h^k$ is the interpolation operator defined by $\p(v) = \sum_{j=1}^{k+1}v(x_j)\varphi_j$. In general, these numerical fluxes depend on the values of the numerical solution from both sides of the interface. Here, we propose a backward (resp. forward) flux to define the trace of $u_h$ (resp. $\phi_h$) at an interface $x_{i\pm\frac{1}{2}}$, i.e.
\begin{align}\label{def_flux}
(\widehat{u}_h)_{i\pm\frac{1}{2}} = (u_h)_{i\pm\frac{1}{2}}^-, \quad (\widehat{\phi}_h)_{i\pm\frac{1}{2}} = (\phi_h)_{i\pm\frac{1}{2}}^+.
\end{align}
It follows that \eqref{semi_disct_scheme_for_1} can be written as: $\forall\ 1\leq i \leq I$ and $\forall\ 1\leq j \leq k+1$
\begin{subequations}\label{semi_disct_scheme_for_2}
\begin{gather}
\begin{aligned}
& \int_{K_{i}} \partial_t u_h^i\, \varphi_j^i dx - \int_{K_{i}} u_h^i\,  \partial_x  \varphi_j^i dx \\
& \qquad + (u_h)^{-}_{i+\frac{1}{2}}\varphi^i_j(x_{i+\frac{1}{2}}) - (u_h)^{-}_{i-\frac{1}{2}}\varphi^i_j(x_{i-\frac{1}{2}}) = \int_{K_i}\phi_h^i\, \varphi_j^i dx \label{semi_disct_scheme_for_u_2}
\end{aligned} \\
\begin{aligned}
& \int_{K_{i}}\partial_t \phi_h^i\, \psi_j^i dx + \int_{{K_i}}\phi_h^i\, \partial_x\psi_j^i dx \\
& \qquad - (\phi_h)^{+}_{i+\frac{1}{2}}\psi^i_j(x_{i+\frac{1}{2}}) + (\phi_h)^{+}_{i-\frac{1}{2}}\psi^i_j(x_{i-\frac{1}{2}}) = \int_{K_i}\p\left(|u_h^i|^p\right)\psi_j^i dx, \label{semi_disct_scheme_for_phi_2}
\end{aligned}
\end{gather}
\end{subequations}
with $u_h^i = {u_h}_{|_{K_i}}$ (resp. $\phi_h^i = {\phi_h}_{|_{K_i}}$) is the restriction of $u_h$ (resp. $\phi$) over the cell $K_i$. Integrating by parts once more, one may write \eqref{semi_disct_scheme_for_2} as: $\forall\ 1\leq i \leq I$ and $\forall\ 1\leq j \leq k+1$
\begin{subequations}\label{semi_disct_for_v}
\begin{gather}
\begin{aligned}
\int_{K_{i}} \left(\partial_t u_h^i+\partial_x u_h^{i}\right) \varphi_j^i\, dx + [u_h]_{i-\frac{1}{2}}\varphi^i_j(x_{i-\frac{1}{2}}) = \int_{K_i}\phi_h^{i}\, \varphi_j^i\, dx \label{semi_disct_for_v1}
\end{aligned} \\
\begin{aligned}
\int_{K_{i}} \left(\partial_t \phi_h^i - \partial_x\phi_h^{i}\right) \psi_j^i\, dx - [\phi_h]_{i+\frac{1}{2}}\psi^i_j(x_{i+\frac{1}{2}}) = \int_{K_i}\p\left(|u_h^{i}|^p\right) \psi_j^i\, dx, \label{semi_disct_for_v2}
\end{aligned}
\end{gather}
\end{subequations}
where $[\cdot]$ denotes the jump at the cell interface. Recall that $u_h$ and $\phi_h$ belong to $V_h^k$, hence one can write 
\begin{align}\label{decomp}
u^{i}_h(x,t) = \sum_{\ell=1}^{k+1} u^{i}_\ell(t) \varphi^i_\ell(x) \quad \text{and} \quad \phi^{i}_h(x,t) = \sum_{\ell=1}^{k+1} \phi^{i}_\ell(t) \psi^i_\ell(x), 
\end{align}
where the coefficients $u_j^{i}$ are called the degrees of freedom and need to be determined at each time. Moreover, since $\p\left(|u_h^{i}|^p\right)$ also belongs to $V_h^k$, and in view of the definition of $V_h^k$, then we have for all $1\leq i \leq I$
\begin{equation}\label{decomp2}
\p\left(|u_h^{i}(x,t)|^p\right) = \sum_{\ell=1}^{k+1} |u^{i}_\ell(t)|^p\, \varphi^i_\ell(x).
\end{equation}
Plugging \eqref{decomp} and \eqref{decomp2} into \eqref{semi_disct_for_v} yields the semi discrete matricial system: $\forall\ t> 0$ and $\forall\ 1\leq i\leq I$
\begin{subequations}\label{semi_disct_scheme}
\begin{gather}
\begin{aligned}
M^i \partial_t U^i_h(t) + R^i\,U^{i}_h(t) + A^i\,U^{i}_h(t) - B^i\,U^{i-1}_h(t) = M^i \Phi^{i}_h(t), 
\end{aligned} \\
\begin{aligned}
M^i \partial_t \Phi^i_h(t) - R^i\,\Phi^{i}_h(t) - C^i\,\Phi^{i+1}_h(t) + D^i\,\Phi^{i}_h(t) = M^i \vert U^{i}_h(t)\vert^{p}, 
\end{aligned}
\end{gather}
\end{subequations}
where $U^i_h = \left(u^i_1,\dots,u^i_{k+1}\right)$, $\Phi^i_h = \left(\phi^i_1,\dots,\phi^i_{k+1}\right)$, $|U^i_h|^p = \left(|u^i_1|^p,\dots,|u^i_{k+1}|^p\right)$ and $\forall\, 1\leq j,\ell\leq k+1$
\begin{equation*}
M^i_{j\ell} = \int_{K_i} \varphi^i_j\ \varphi^i_\ell\ dx,\quad R^i_{j\ell} = \int_{K_i}  \varphi^i_j\ \partial_x\varphi^i_\ell\ dx,
\end{equation*}
\begin{equation*}
A^i_{j\ell} = \varphi^i_j(x_{i-\frac{1}{2}})\ \varphi^{i}_\ell(x_{i-\frac{1}{2}}), \quad B^i_{j\ell} = \varphi^i_j(x_{i-\frac{1}{2}})\ \varphi^{i-1}_\ell(x_{i-\frac{1}{2}}),
\end{equation*}
and
\begin{equation*}
C^i_{j\ell} = \varphi^i_j(x_{i+\frac{1}{2}})\ \varphi^{i+1}_\ell(x_{i+\frac{1}{2}}), \quad D^i_{j\ell} = \varphi^i_j(x_{i+\frac{1}{2}})\ \varphi^{i}_\ell(x_{i+\frac{1}{2}}).
\end{equation*}
For the boundary conditions, we set $U^0_h(t) := U^{I}_h(t)$ and $\Phi^{I+1}_h(t) := \Phi^1_h(t)$ for all $t\geq 0$.
%
%--------------------------------------------------------------------------------------------
%
\subsection{Time discretization}
A fully discrete scheme of \eqref{semi_disct_scheme} can be derived using an approximation of the time derivative $\partial_t U_h$ and $\partial_t \Phi_h$. Here, we used the explicit forward Euler method with non constant time step. Let $\Delta t^0$, $\Delta t^1$, $\ldots$ be positive constants and set
\begin{align}\label{delta_t}
t^0 = 0,\quad t^n = \sum_{\ell=0}^{n-1} \Delta t^\ell = t^{n-1}+ \Delta t^{n-1} \quad (n\geq1).
\end{align}
Then, we approximate the time derivative of $U_h$ and $\Phi_h$ at time $t^n$ as follows
\begin{equation*}
\partial_t U_h(t^n) \approx \dfrac{U_h^{n+1} - U_h^n}{\Delta t^n} \quad \text{and}\quad \partial_t \Phi_h(t^n) \approx \dfrac{\Phi_h^{n+1} - \Phi_h^n}{\Delta t^n}
\end{equation*}
where $U_h^n$ (resp. $\Phi_h^n$) is the value of $U_h$ (resp. $\Phi_h$) at time $t^n$. %Following \cite{Sasaki}, we choose $\Delta t^n$ as given by \eqref{time_step}. 
The fully discrete DG scheme for the non linear wave equation \eqref{eqq1} is then given by: $\forall\, n\geq 0$, $\forall\ 1\leq i \leq I$ and $\forall\ 1\leq j \leq k+1$
\begin{subequations}\label{disct_for_v}
\begin{gather}
\begin{aligned}
\int_{K_{i}} \left(\dfrac{u_h^{i,n+1} - u_h^{i,n}}{\Delta t^n} + \partial_x u_h^{i,n}\right) \varphi_j^i\, dx + [u_h^n]_{i-\frac{1}{2}}\varphi^i_j(x_{i-\frac{1}{2}}) = \int_{K_i}\phi_h^{i,n}\, \varphi_j^i\, dx \label{disct_for_v1}
\end{aligned} \\
\begin{aligned}
\int_{K_{i}} \left(\dfrac{\phi_h^{i,n+1} - \phi_h^{i,n}}{\Delta t^n} - \partial_x\phi_h^{i,n}\right) \psi_j^i\, dx - [\phi_h^n]_{i+\frac{1}{2}}\psi^i_j(x_{i+\frac{1}{2}}) = \int_{K_i}\p\left(|u_h^{i,n+1}|^p\right)\psi_j^i\, dx, \label{disct_for_v2}
\end{aligned}
\end{gather}
\end{subequations}
with the initial conditions $(u_h^{i,0},\phi_h^{i,0}) = (\p u_0^i,\p \phi_0^i)$ and the periodic boundary conditions $(u_h^{1,n},\phi_h^{1,n}) = (u_h^{I,n},\phi_h^{I,n})$. Equivalently, the system \eqref{disct_for_v} writes in matricial form: $\forall\ n\geq 0$ and $\forall\ 1\leq i\leq I$
\begin{subequations}\label{disct_scheme}
\begin{gather}
\begin{alignat}{1}
& M^i \dfrac{U^{i,n+1}_h-U^{i,n}_h}{\Delta t^n}  + \left(R^i+A^i\right)U^{i,n}_h - B^i\,U^{i-1,n}_h = M^i \Phi^{i,n}_h,\\ & \notag \\
& M^i \dfrac{\Phi^{i,n+1}_h-\Phi^{i,n}_h}{\Delta t^n} - \left(R^i-D^i\right)\Phi^{i,n}_h - C^i\,\Phi^{i+1,n}_h = M^i \vert U^{i,n+1}_h\vert^{p}, \\ & \notag \\
& U^{i,0}_h = \p u_0^i,\ \ \Phi^{i,0}_h = \p \phi_0^i, \\ & \notag \\
& U^{0,n}_h := U^{I,n}_h,\ \ \Phi^{I+1,n}_h := \Phi^{1,n}_h.
\end{alignat}
\end{gather}
\end{subequations}
Let us notice that scheme \eqref{disct_for_v} or equivalently \eqref{disct_scheme} is fully explicit in time. % since the term $u_h^{n+1}$ involved in the second equation is computed in the first equation. 
 This is of major advantage since neither matrix inversions nor implicit nonlinear computations have to be performed in order to evaluate the numerical solution at each time step.
%
%%%%%%%%%%%%%%%%%%%%%%%%%%%%%%%%%%%%%%%%%%%%%%%%%%%%%%%%%%%%%%%%%%%%%%%%%%%%%%%%%%%%%%%%%%%%%
%
\section{Study of the DG scheme}\label{Sec_Stability}
We prove in this section the consistency and the local stability of the DG scheme.
%
%--------------------------------------------------------------------------------------------
%
\subsection{Consistency}
\begin{lem}
The DG scheme \eqref{disct_for_v} is consistent with the system \eqref{eq2}.
\end{lem}
\begin{proof}
It is obvious from \eqref{def_flux} that the numerical fluxes are monotone and thus consistent \cite{Cock_shu_89}. Our purpose now is to prove that the approximation of the nonlinear term is also consistent with the original system \eqref{eq2}. We shall assume that the solution $u\in C^2([0,T_\infty),H^{m+1}(a,b))$, $m\geq 1$, and thus the jumps $[u]_{i+\frac{1}{2}}$ and $[\phi]_{i+\frac{1}{2}}$ vanish over the interfaces $x_{i+\frac{1}{2}}$ for all $0\leq i \leq I$ and for all time $t$. % Substitute $u^{i,n}$ by $u(x,t^n)$ in \eqref{disct_for_v1}-\eqref{disct_for_v2} yields
 Denote
\begin{align}\label{r_i^n}
r^n & := \sum_{i=1}^I \int_{K_{i}} \left(\dfrac{u(x,t^{n+1}) - u(x,t^{n})}{\Delta t^n} + \partial_x u(x,t^{n})\right) \varphi^i(x)\, dx \\
& \qquad + \sum_{i=1}^I \underbrace{[u(\cdot,t^n)]_{i-\frac{1}{2}}}_{=0}\varphi^i(x_{i-\frac{1}{2}}) - \sum_{i=1}^I \int_{K_i} \phi(x,t^n)\, \varphi^i(x)\, dx \nonumber
\end{align}
and
\begin{align}\label{s_i^n}
s^n & := \sum_{i=1}^I \int_{K_{i}} \left(\dfrac{\phi(x,t^{n+1}) - \phi(x,t^{n})}{\Delta t^n} - \partial_x\phi(x,t^{n})\right) \psi^i(x)\, dx \\
& \qquad - \sum_{i=1}^I \underbrace{[\phi(\cdot,t^n)]_{i+\frac{1}{2}}}_{=0}\psi^i(x_{i+\frac{1}{2}}) - \sum_{i=1}^I \int_{K_i}\p\left(|u(x,t^{n+1})|^p\right)\psi^i(x)\, dx. \nonumber
\end{align}
It follows by \eqref{eq2} and using a first order Taylor expansion in \eqref{r_i^n} that
\begin{equation*}
|r^n| \leq C_1\Delta t^n \quad \forall\ n\geq 0
\end{equation*}
with $C_1>0$ is independent of $\Delta t^n$. Similarly, we have using a second order Taylor series in \eqref{s_i^n}
\begin{align*}
s^n & = \sum_{i=1}^I \int_{K_{i}} \Delta t^n \partial_{tt}\phi(x,\xi^n)\psi^i(x)\,dx \\
& \qquad + \sum_{i=1}^I \int_{K_i}\left(|u(x,t^{n+1})|^p - \p\left(|u(x,t^{n+1})|^p\right)\right)\psi^i(x)\, dx.
\end{align*}
Using the classical estimate (see e.g. \cite[Theorem 1.103]{Ern})
\begin{equation}\label{u-pu}
\|v - \p v\|_{L^\infty(K)} \leq \tilde{C} h^{m} \quad \text{for any}\ v\in H^{m+1}(K)
\end{equation}
we deduce
\begin{equation*}
|s^n| \leq C_2 \Delta t^n + C_3h^{m} \quad \forall\ n\geq 0.
\end{equation*}
with $C_2$ and $C_3$ are positive constants independent of $\Delta t^n$ and $h$. This concludes the consistency of the proposed DG scheme.
\end{proof}
%
%--------------------------------------------------------------------------------------------
%
\subsection{Positivity and local stability}
For $u_h\in V_h^k$, we define the norm
$$ \|u_h\|_\infty := \|U_h\|_\infty = \max_{1\leq i \leq I} \|U_h^i\|_\infty = \max_{1\leq i \leq I}\max_{1\leq j \leq k+1} |u_j^i|,$$
where the $u^i_j$ are the coordinates of $u_h$ in the Lagrange polynomial basis.
\begin{prop}\label{prop_stability}
Let $\sigma>0$ and $\nu> 0$ be arbitrary real numbers and set 
\begin{equation}\label{time_step}
% \Delta t^n = h^{1+\sigma}\min\left(1,\frac{1}{\|u_h^n\|_\infty^{1+\nu}},\frac{1}{\|\phi_h^n\|_\infty^{1+\nu}}\right).
\Delta t^n = h^{1+\sigma}\min\left(1,\frac{1}{\|u_h^n\|_\infty^{1+\nu}}\right).
\end{equation}
Suppose  the initial conditions satisfy $\min(u_0,\phi_0)> \mu\geq 0$. Then, for any $N\in\mathds{N}$, there exists a constant $h_{N} >0$ depending on $N$, $u_0$ and $\phi_0$ such that for all $h\in(0,h_{N}]$,
\begin{equation}\label{u^n>0}
U_h^n > \mu \ \text{ and }\ \Phi_h^n >\mu \quad \forall\ 1\leq n\leq N.
\end{equation}
(the inequalities are element-wise). In addition, if $\sum_{n\geq 0}\frac{1}{\|u_h^n\|_\infty^\nu} + \frac{1}{\|\phi_h^n\|_\infty^\nu}<\infty$, % we set $\Delta t^n = h^{1+\sigma}\min\left(1,\frac{1}{(n+1)^\nu\|u^n\|_\infty},\frac{1}{(n+1)^\nu\|\phi^n\|_\infty}\right)$ with $\nu>1$, 
then \eqref{u^n>0} holds for $h_N=h_*$ independent of $N$.
\end{prop}
\begin{proof}
We proceed by induction on $n$. Since $u_0> \mu\geq 0$ (resp. $\phi_0> \mu\geq 0$) then $u^{i,0}_j = u_{0|_{K_i}}(x^i_j)> \mu$ (resp. $\phi^{i,0}_j = \phi_{0|_{K_i}}(x^i_j)> \mu$) and hence \eqref{u^n>0} holds true for $n=0$. Let $N\in \IN$ and suppose \eqref{u^n>0} is valid for all $0\leq n \leq N-1$, then $u^{i,n}_j>\mu$ and $\phi^{i,n}_j> \mu$ for all $1\leq i \leq I$ and all $1\leq j \leq k+1$. Moreover, equation \eqref{disct_scheme} reads
\begin{equation*}
U^{i,n+1}_h = U^{i,n}_h + \dfrac{\Delta t^n}{h}\left(E\,U^{i,n}_h + F\,U^{i-1,n}_h\right) + \Delta t^n \Phi^{i,n}_h
\end{equation*}
with $E = h (M^i)^{-1}(R^i+A^i)$ and $F = -h(M^i)^{-1}B^i$ are constant matrices (i.e. do not depend on $h$), and satisfy
\begin{equation}\label{E+F}
\sum_{\ell=1}^{k+1} E_{j\ell} + F_{j\ell} = 0 \quad \forall \ 1\leq j \leq k+1.
\end{equation}
(see \ref{app_mat} for details). Denote $x^+ = \max(x,0)$ and $x^- = \min(x,0)$ for any $x\in \IR$, then we obtain for $1\leq i \leq I$ and $1\leq j \leq k+1$
\begin{align*}
u^{i,n+1}_j & = u^{i,n}_j + \dfrac{\Delta t^n}{h} \sum_{\ell=1}^{k+1} \left(E_{j\ell}\,u^{i,n}_\ell + F_{j\ell}\,u^{i-1,n}_\ell\right) + \Delta t^n \phi^{i,n}_j \\
& = u^{i,n}_j + \dfrac{\Delta t^n}{h} \left(\sum_{\ell=1}^{k+1} \left(E_{j\ell}^+\,u^{i,n}_\ell + F_{j\ell}^+\,u^{i-1,n}_\ell\right) + \sum_{\ell=1}^{k+1} \left(E_{j\ell}^-\,u^{i,n}_\ell + F_{j\ell}^-\,u^{i-1,n}_\ell\right)\right) + \Delta t^n \phi^{i,n}_j \\
& \geq \min_{i,j} u^{i,n}_j + \dfrac{\Delta t^n}{h} \left(\sum_{\ell=1}^{k+1} \left(E_{j\ell}^+ + F_{j\ell}^+\right) \min_{i,\ell}u^{i,n}_\ell + \sum_{\ell=1}^{k+1} \left(E_{j\ell}^- + F_{j\ell}^-\right)\max_{i,\ell}u^{i,n}_\ell\right) \\
& = \min_{i,j} u^{i,n}_j + \dfrac{\Delta t^n}{h}\left(\sum_{\ell=1}^{k+1} \left(E_{j\ell}^+ + F_{j\ell}^+\right)\right)\left(\min_{i,\ell}u^{i,n}_\ell - \max_{i,\ell}u^{i,n}_\ell\right)
\end{align*}
where the last equality holds in view of \eqref{E+F}. Let $\alpha_n = \rho\frac{\Delta t^n}{h}$ with 
\begin{equation}\label{rhoo}
\rho=\dis{\min_{1\leq j\leq k+1} \sum_{\ell=1}^{k+1} \left(E_{j\ell}^+ + F_{j\ell}^+\right)},
\end{equation}
and denote $\dis{v_n = \min_{i,j} u^{i,n}_j}$ and $\dis{w_n = \max_{i,j} u^{i,n}_j = \|u_h^n\|_\infty}$, then we have
\begin{equation}\label{vn+1=vn-wn}
v_{n+1} \geq v_n + \alpha_n (v_n-w_n).
\end{equation}
%Since $\Delta t^n \leq \dfrac{h^{1+\sigma}}{\|u^n\|_\infty}$ then $\alpha_n \leq \dfrac{\rho h^\sigma}{\|u^n\|_\infty}$. Hence \eqref{vn+1=vn-wn} implies
%$$ v_{n+1} \geq (1 + \alpha_n) v_n - \rho h^\sigma.$$
A straightforward induction on $n$ shows that
\begin{align}\label{ineq_vn+1}
v_{n+1} & \geq \left(\prod_{m=0}^{n}(1+\alpha_{m})\right) v_0 - \sum_{\ell=0}^{n} \left(\prod_{m=\ell+1}^{n}(1+\alpha_{m})\right)\alpha_\ell w_\ell \nonumber \\
& \geq \left(\prod_{m=0}^{n}(1+\alpha_{m})\right)\left(v_0-\dfrac{\rho}{h}\sum_{\ell=0}^{n}\Delta t^\ell\, \|u^\ell\|_\infty\right)
\end{align}
Now, if $\Delta t^\ell \leq \frac{h^{1+\sigma}}{\|u_h^\ell\|_\infty^{1+\nu}} \leq \frac{h^{1+\sigma}}{\|u_h^\ell\|_\infty}$ then the inequality \eqref{ineq_vn+1} implies $\forall\ 0\leq n\leq N$ $$v_{n+1}> v_0-N\rho h^\sigma.$$
Hence, if $h\leq h_N:=\left(\dfrac{v_0-\mu}{N\rho}\right)^{1/\sigma}$, then $v_{n+1}>\mu$ and by definition of $v_n$, we obtain $U_h^{n+1}> \mu$. Moreover, if $\sum_{n\geq 0} \frac{1}{\|u_h^n\|_\infty^\nu} < \infty$ then $S = \sum_{n\geq 0} \Delta t^n \|u_h^n\|_\infty < \infty$ and \eqref{ineq_vn+1} implies $v_{n+1} > v_0-S\rho h^\sigma$. Take $h_* = \left(\dfrac{v_0-\mu}{S\rho}\right)^{1/\sigma}$ yields the result. The proof for $\Phi_h^n$ is similar.% Finally, it follows from the induction hypothesis that if $h\in (0,h_N]$ with $$h_N = \dfrac{v_0}{\dis{N\gamma \min_{1\leq j\leq k+1} \sum_{\ell=1}^{k+1} \left(E_{j\ell}^+ + F_{j\ell}^+\right)}},$$ then $u^n\geq 0$ for all $0\leq n \leq N$.
\end{proof}
\begin{rmq}\label{rmq_0k7}
Equation \eqref{u^n>0} states that the discrete maximum principle is fulfilled for $\IP_0$ and $\IP_1$ approximations. %$k=0$ or $k=1$.
% When $k=0$ or $k=1$, equation \eqref{u^n>0} is equivalent to $u^n_h>\mu$ and $\phi^n_h>\mu$ for all $1\leq n\leq N$. However, the discrete maximum principle is no longer fulfilled if $k\geq 2$ since the numerical solution may oscillate between its degrees of freedom. % To ensure a discrete maximum principle, %overcome such an issue, one may choose modal functions \cite{Piperno} as a basis of $V_h^k$ instead of the nodal functions. % However, the projection of the nonlinear term $|u_h|^p$ into $V_h^k$ would introduce more computational complexity to the system \eqref{disct_scheme}.
% When $k=0$ or $k=1$, the discrete maximum principle is satisfied, while This is no longer guaranteed if $k\geq 2$ since the numerical solution may oscillate between its degrees of freedom.
\end{rmq}
\begin{thm}\label{thm_stab}
Let $\Delta t^n$ be given by \eqref{time_step}, and let $\Lambda_\infty = \| u^0_h \|_\infty + \| \phi^0_h \|_\infty$. Then, for any $N\in\mathds{N}$ there exists a constant $h_{N,\Lambda_\infty} >0$ depending only on $N$ and $\Lambda_\infty$ such that if $h\in(0,h_{N,\Lambda_\infty}]$, then
\begin{equation}\label{theo_stability}
 \sup_{1\leq n\leq N}(\| u^n_h \|_\infty + \| \phi^n_h \|_\infty)\leq 2\Lambda_\infty .
\end{equation}
\end{thm}
\begin{proof}
First, we rewrite the scheme \eqref{disct_scheme} as 
\begin{align}\label{u_phi}
\left\{
\begin{array}{lll}
U^{n+1}_h = & \! M_n U^n_h + \Delta t^n \Phi^n_h & \\
% & & \quad n\geq0\\
\Phi^{n+1}_h = & \! N_n \Phi^n_h + \Delta t^n f(U^{n+1}_h) &
\end{array}
\right.
\end{align}
where
\begin{equation*}
M_n = \begin{pmatrix}
\mathcal{M}_A & 0 & \dots & 0 & \mathcal{M}_B \\
\mathcal{M}_B & \mathcal{M}_A & 0 & \dots & 0 \\
0 & \mathcal{M}_B & \mathcal{M}_A & \ddots & \vdots\\
\vdots & \ddots & \ddots & \ddots & 0 \\
0 & \dots & 0 & \mathcal{M}_B & \mathcal{M}_A
\end{pmatrix} \quad \text{and }\ 
N_n = \begin{pmatrix}
\mathcal{N}_D & \mathcal{N}_C & 0 & \dots & 0 \\
0 & \mathcal{N}_D & \mathcal{N}_C & \ddots & \vdots \\
\vdots & \ddots & \ddots & \ddots & 0\\
0 & \dots & 0 & \mathcal{N}_D & \mathcal{N}_C \\
\mathcal{N}_C & 0 & \dots & 0 & \mathcal{N}_D
\end{pmatrix}
\end{equation*}
with 
\begin{align*}
\mathcal{M}_A & = I_{k+1} - \Delta t^n M^{-1}(R+A), \quad  \mathcal{M}_B = \Delta t^n M^{-1}B, \\
\mathcal{N}_D & = I_{k+1} - \Delta t^nM^{-1}(D-R), \quad  \mathcal{N}_C = \Delta t^n M^{-1}C,
\end{align*}
and
\begin{align*}
f(v) = (\vert v_1\vert^p,\dots,\vert v_I\vert^p)^{T}\quad \text{for} \quad v = (v_1,\dots,v_I)^T.
\end{align*}
Now, we prove \eqref{theo_stability} by induction on $n$. Let $N\in \IN$ and assume that 
$$ \|U^n_h\|_\infty + \|\Phi^n_h\|_\infty \leq 2\Lambda_\infty\quad \forall\ 0\leq n\leq N-1.$$ Using \eqref{u_phi}, we may rewrite $U^{n+1}$ and $\Phi^{n+1}$ as
\begin{align}
U^{n+1}_h & = M_n\dots M_0\ U^0_h + \sum_{j=0}^{n}\Delta t^{n-j} M_n\dots M_{n-j+1} \Phi^{n-j}_h, \label{u=MMM}\\
\Phi^{n+1}_h & = N_n\dots N_0\ \Phi^0_h + \sum_{j=0}^{n}\Delta t^{n-j}N_n\dots N_{n-j+1} \ f(U^{n-j+1}_h) \label{phi=MMM}.
\end{align}
We have the following result.
\begin{lem}\label{lemma1}
$\|M_n\|_\infty = \|N_n\|_\infty \leq 1+2\rho\dfrac{\Delta t^n}{h}$.
\end{lem}
\begin{proof}
See \ref{app_proof_lem}.
\end{proof}
It follows by the induction hypothesis
\begin{align*}
\|U^{n+1}_h\|_\infty & \leq \prod_{\ell=0}^n \left(1+2\rho\frac{\Delta t^\ell}{h}\right)\|U^0_h\|_\infty + h^{1+\sigma} \sum_{j=0}^{n} \prod_{\ell=0}^{j-1} \left(1+2\rho\frac{\Delta t^{n-\ell}}{h}\right) \|\Phi^{n-j}_h\|_\infty \\
& \leq \prod_{\ell=0}^n \left(1+2\rho h^\sigma\right)\left(\|U^0_h\|_\infty + 2\Lambda_\infty(n+1)h^{1+\sigma}\right)\\
& = \left(1+2\rho h^\sigma\right)^{n+1}\left(\|U^0_h\|_\infty + 2\Lambda_\infty(n+1)h^{1+\sigma}\right)
\end{align*}
where $\rho$ is given by \eqref{rhoo}. It follows that $\forall \ 0\leq n \leq N-1$
\begin{equation}\label{estim_Un+1}
\|U^{n+1}_h\|_\infty \leq \left(1+2\rho h^\sigma\right)^N\left(\|U^0_h\|_\infty + 2\Lambda_\infty Nh^{1+\sigma} \right).
\end{equation}
Similarly, we obtain from \eqref{phi=MMM}
\begin{align*}
\|\Phi^{n+1}_h\|_\infty & \leq \prod_{\ell=0}^n \left(1+2\rho\frac{\Delta t^\ell}{h}\right)\|\Phi^0_h\|_\infty + h^{1+\sigma}\sum_{j=1}^{n} \prod_{\ell=0}^{j-1} \left(1+2\rho\frac{\Delta t^{n-\ell}}{h}\right) \|U^{n-j+1}_h\|_\infty^p \\
& \qquad + h^{1+\sigma}\|U^{n+1}_h\|_\infty^p\\
& \leq \left(1+2\rho h^\sigma\right)^{n+1}\left(\|\Phi^0_h\|_\infty + (2\Lambda_\infty)^p n h^{1+\sigma}\right) \\
& \qquad + h^{1+\sigma}\left(1+2\rho h^\sigma\right)^{p(n+1)}\left(\|U^0_h\|_\infty + 2\Lambda_\infty(n+1)h^{1+\sigma}\right)^p.
\end{align*}
Using the identity $(x+y)^r \leq 2^{r-1}(x^r+y^r)$ for any non negative reals $x$ and $y$ and any $r\geq 1$, we obtain $\forall\ 0\leq n \leq N-1$
\begin{align}\label{estim_Phin+1}
\|\Phi^{n+1}_h\|_\infty & \leq \left(1+2\rho h^\sigma\right)^{N}\left(\|\Phi^0_h\|_\infty + (2\Lambda_\infty)^p N h^{1+\sigma} \right) \nonumber\\
& \qquad + 2^{p-1}h^{1+\sigma}\left(1+2\rho h^\sigma\right)^{pN}\left(\|U^0_h\|_\infty^p + (2\Lambda_\infty Nh^{1+\sigma})^p\right).
\end{align}
It follows by \eqref{estim_Un+1} and \eqref{estim_Phin+1}
\begin{align*}%\label{estim_forcv}
\|U^{n+1}_h\|_\infty+\|\Phi^{n+1}_h\|_\infty & \leq \left(1+2\rho h^\sigma\right)^{N}\Lambda_\infty + N h^{1+\sigma}\left(1+2\rho h^\sigma\right)^{N}\left(2\Lambda_\infty + (2\Lambda_\infty)^p\right) \nonumber\\
& \quad +  2^{p-1}h^{1+\sigma}\left(1+2\rho h^\sigma\right)^{pN}\Lambda_\infty^p\left(1 + (2Nh^{1+\sigma})^p\right).
\end{align*}
Set 
\begin{align*}
h_{N,\Lambda_\infty} & = \min\left\{\left(\dfrac{(\frac{3}{2})^\frac{1}{N}-1}{2\rho}\right)^\frac{1}{\sigma},\ \dfrac{\Lambda_\infty}{\left[12N\Lambda_\infty(1+(2\Lambda_\infty)^{p-1})\right]^{\frac{1}{1+\sigma}}},\right.\\
& \qquad \qquad \left. \dfrac{\Lambda_\infty}{\left[4\left(3\Lambda_\infty\right)^p\left(1+\frac{\Lambda_\infty^{p\sigma}}{6^p(1+(2\Lambda_\infty)^{p-1})^p}\right)\right]^{\frac{1}{1+\sigma}}}\right\},
\end{align*}
then one can check that $\forall\ h\in(0,h_{N,\Lambda_\infty}]$ we have
$$\|U^{n+1}_h\|_\infty+\|\Phi^{n+1}_h\|_\infty \leq \dfrac{3\Lambda_\infty}{2} + \dfrac{\Lambda_\infty}{4}+ \dfrac{\Lambda_\infty}{4} = 2\Lambda_\infty.$$
% The proof is completed.
\end{proof}
%
%Let $\alpha_j:= \int_{-1}^1\varphi_j\,dx \neq 0$ for all $1\leq j \leq k+1$. Then, for any $u_h\in V_h^k$, the quantity
%\begin{equation}\label{norm_L1}
%\|u_h\|_1 = \dfrac{1}{b-a}\sum_{i=1}^I \|u^i_h\|_1 \ \text{ with }\  \|u^i_h\|_1:= \frac{h}{2}\sum_{j=1}^{k+1}|\alpha_j u^i_j|
%\end{equation}
%defines a norm which is equivalent (since in finite dimension) to $\|u_h\|_\infty$. Actually, we have 
%\begin{equation}\label{kappa}
%\kappa \|u_h\|_\infty \leq \|u_h\|_1 \leq \|u_h\|_\infty, 
%\end{equation}
%with $\kappa = \frac{1}{2}\min_{1\leq j \leq k+1}|\alpha_j|$. Then, a direct consequence of Proposition \ref{prop_stability} and Theorem \ref{thm_stab} is the following.
%\begin{cor}\label{cor_thm}
%Suppose the initial conditions satisfy $\min(u_0,\phi_0)>\mu \geq 0$, and denote $\Lambda_1 = \| u^0_h \|_1 + \| \phi^0_h \|_1$. Then, for any $N\in\mathds{N}$ there exists a constant $h_{N,\Lambda_1} >0$ depending only on $N$ and $\Lambda_1$ such that if $h\in(0,h_{N,\Lambda_1}]$, then for all $1\leq n\leq N$
%$$ \min(\| u^n_h \|_1 ,\, \| \phi^n_h \|_1) > \kappa \mu\ \text{ and }\ \| u^n_h \|_1 + \| \phi^n_h \|_1 \leq \frac{2\Lambda_1}{\kappa}.$$
%\end{cor}
%
%%%%%%%%%%%%%%%%%%%%%%%%%%%%%%%%%%%%%%%%%%%%%%%%%%%%%%%%%%%%%%%%%%%%%%%%%%%%%%%%%%%%%%%%%%%%%
%
\section{Numerical blow-up}% and convergence of the DG scheme}
\label{sec_blowup}
In this section, we prove that the numerical blow-up time converges toward the exact blow-up time if the discrete solution $u_h$ tends toward the exact solution $u$ as $h$ tends to zero. The following functional will be useful. 
\begin{align}\label{funct_K}
K(u(t)) := \frac{1}{b-a}\int_a^b u(x,t)dx.
\end{align}
% plays an important role.
\begin{prop}\label{prop_Ku_to_inf}\cite{Sasaki} Assume that 
\begin{align*}
\alpha = K(u_0)\geq 0,\, \,  \beta = K(u_1) >0.
\end{align*}
Then, the solution $u$ of \eqref{eqq1} blows up in finite time $T_\infty \in (0,\infty)$.
\end{prop}
\begin{defn}\label{def_blow}
We define the numerical blow-up time by $$T(h)=\lim_{n\longrightarrow\infty} t^n = \sum_{n=0}^\infty\Delta t^n.$$ We say that the numerical solution blows up if $$\lim_{n\to\infty}\Vert u^n_h\Vert_{L^\infty(a,b)} = \lim_{t^n\to T(h)}\Vert u^n_h\Vert_{L^\infty(a,b)} = \infty.$$ Moreover, we say that the numerical solution blows up in finite time if \mbox{$T(h)<\infty$.}
\end{defn}
\begin{prop}\label{K_h}
Let $0\leq k \leq 7$ and let $(u_h^n,\phi_h^n)$ be the solution of \eqref{disct_for_v}. Define
\begin{equation}\label{the_fonc_K_h}
K_h(u^n_h) = \frac{1}{b-a}\sum_{i=1}^I\int_{K_i}u^{i,n}_h(x) dx, 
\end{equation}
and suppose $\beta_h := K_h(u^1_h) > 0$ and $\alpha_h := K_h(u^0_h)\geq 0$. Then $(K_h(u^n_h))_n$ is a strictly increasing unbounded sequence and for all $n\geq 0$
\begin{equation*}
\left(\frac{K_h(u^{n+1}_h)-K_h(u^n_h)}{\Delta t^n}\right)^2\geq \frac{\lambda}{p+1}\left(K_h(u^n_h)\right)^{p+1} + \gamma_h\geq 0
\end{equation*}
where $$\gamma_h = \left(\frac{\beta_h - \alpha_h}{\Delta t^0}\right)^2 - \frac{\lambda}{p+1}\alpha_h^{p+1}$$
and $\lambda>0$ is a constant independent of $h$.
\end{prop}
\begin{proof}
Recall that the scheme \eqref{disct_for_v1}-\eqref{disct_for_v2} is equivalent to equations \eqref{semi_disct_scheme_for_u_2}-\eqref{semi_disct_scheme_for_phi_2}. Then, take $\varphi^i_j \equiv 1$ in \eqref{semi_disct_scheme_for_u_2} yields% and sum up over $\ell$ from $1$ to $k+1$ yields % 
\begin{align*}
\int_{K_i}\dfrac{u^{i,n+1}_h-u^{i,n}_h}{\Delta t^n}dx + u_h^{i,n}(x_{i+\frac{1}{2}}) - u_h^{i-1,n}(x_{i-\frac{1}{2}}) =\int_{K_i} \phi^{i,n}_h dx.
\end{align*}
Sum up over $i=1,\dots I$ and use the periodic boundary condition,
\begin{equation}\label{ineq_Kh_u}
\dfrac{K_h(u^{n+1}_h) - K_h(u^n_h)}{\Delta t^n} = K_h(\phi^n_h)\quad \forall\ n\geq 0.
\end{equation}
In particular 
\begin{equation}\label{ineq_Kh_phi0}
\frac{K_h(u^1_h) - K_h(u^0_h)}{\Delta t^0} = K_h(\phi^0_h)>0.
\end{equation}
Similarly, we have by \eqref{semi_disct_scheme_for_phi_2}
\begin{align*}
\int_{K_i}\dfrac{\phi^{i,n+1}_h-\phi^{i,n}_h}{\Delta t^n}dx - \phi_h^{i+1,n}(x_{i+\frac{1}{2}}) + \phi_h^{i,n}(x_{i-\frac{1}{2}}) = \int_{K_i} \p\left(|u_h^{i,n+1}|^p\right) dx,
\end{align*}
and hence %Sum up over $i=1,\dots I$ and use the periodic boundary condition,
\begin{equation*}
\dfrac{K_h(\phi^{n+1}_h) - K_h(\phi^n_h)}{\Delta t^n} = K_h(\p(|u^{n+1}_h|^p))\quad \forall\ n\geq 0.
\end{equation*}
At this stage, we need the following technical lemma.
\begin{lem}\label{lem_Kh>Khp}
Let $0\leq k \leq 7$. Then, there exists $\lambda>0$ independent of $h$ such that
$$K_h\big{(}\p(|u^{n+1}_h|^p)\big{)} \geq \lambda\left(K_h(u^{n+1}_h)\right)^p .$$
\end{lem}
\begin{proof}
See \ref{app_lem_lambda}.
\end{proof}
Thus, we have
\begin{equation}\label{ineq_Kh_phi}
\frac{K_h(\phi^{n+1}_h) - K_h(\phi^n_h)}{\Delta t^n} \geq \lambda\left(K_h(u^{n+1}_h)\right)^p.
\end{equation}
Using \eqref{ineq_Kh_u} and \eqref{ineq_Kh_phi}, one can easily show by induction on $n$ that $K_h(u^n_h)$ and $K_h(\phi^n_h)$ are non negative for all $n$. Now, combining \eqref{ineq_Kh_u}, \eqref{ineq_Kh_phi0} and \eqref{ineq_Kh_phi} yields 
\begin{align}
\frac{K_h(u^{n+2}_h) - K_h(u^{n+1}_h)}{\Delta t^{n+1}}&\geq \frac{K_h(u^{n+1}_h)-K_h(u^n_h)}{\Delta t^n} + \lambda \Delta t^n(K_h(u^{n+1}_h))^p\label{croi_of_K_h}\\
&\geq \frac{K_h(u^{1}_h)-K_h(u^0_h)}{\Delta t^0}+ \lambda \sum_{k=0}^{n} \Delta t^k (K_h(u^{k+1}_h))^p \label{croi}  \\ 
& > 0 \quad \forall\ n \geq 0. \nonumber %\label{croi_of_K} 
\end{align}
Consequently, $\left(K_h(u^n_h)\right)_n$ is a strictly increasing sequence. Now, we again make use of \eqref{croi_of_K_h} to obtain
\begin{align*}
& \left(\frac{K_h(u^{n+2}_h)-K_h(u^{n+1}_h)}{\Delta t^{n+1}}\right)^2 \\
& \qquad \geq \frac{K_h(u^{n+1}_h)-K_h(u^{n}_h)}{\Delta t^{n}}\left(\frac{K_h(u^{n+1}_h)-K_h(u^n_h)}{\Delta t^n} +\lambda \Delta t^n(K_h(u^{n+1}_h))^p\right)\\
& \qquad = \left(\frac{K_h(u^{n+1}_h)-K_h(u^n_h)}{\Delta t^n}\right)^2 + \lambda \big(K_h(u^{n+1}_h)-K_h(u^n_h)\big)(K_h(u^{n+1}_h))^p.
\end{align*}
A straightforward induction implies
\begin{align*}
&\left(\frac{K_h(u^{n+2}_h)-K_h(u^{n+1}_h)}{\Delta t^{n+1}}\right)^2 \\
&\qquad \geq \lambda \sum_{k=0}^n\Big(K_h(u^{k+1}_h)-K_h(u^k_h)\Big)\left(K_h(u^{k+1}_h)\right)^p + \left(\frac{K_h(u^1_h)-K_h(u^0_h)}{\Delta t^0}\right)^2 \\
&\qquad \geq \lambda \int_{\alpha_h}^{K_h(u^{n+1}_h)} z^p dz + \left(\dfrac{\beta_h-\alpha_h}{\Delta t^0}\right)^2\\
&\qquad = \dfrac{\lambda}{p+1}\Big( (K_h(u^{n+1}_h))^{p+1} - \alpha_h^{p+1} \Big) + \left(\dfrac{\beta_h-\alpha_h}{\Delta t^0}\right)^2. % \label{end_of_proof}
\end{align*}
Moreover, since $K_h(u^n_h)$ is increasing in $n$, then $\dfrac{\lambda }{p+1}\Big( (K_h(u^{n+1}_h))^{p+1} - \alpha_h^{p+1} \Big) + \left(\dfrac{\beta_h-\alpha_h}{\Delta t^0}\right)^2$ is non negative. Finally, assume $(K_h(u^n_h))_n$ is bounded, then it is convergent. Hence, we can extract a sub-sequence $(u^{n_\ell}_h)_{n_\ell}$ of $(u_h^n)_n$ which converges a.e., and thus it is bounded. We deduce from \eqref{time_step} that $\Delta t^{n_\ell} \not \to 0$ as $n_\ell$ goes to infinity, and using \eqref{ineq_Kh_phi0} and \eqref{croi} we obtain
\begin{align*}
0 < \Delta t^{n_\ell+1}\,K_h(\phi^0_h) \leq K_h(u^{n_\ell+2}_h)-K_h(u^{n_\ell+1}_h).
\end{align*}
Take the limit when $n_\ell$ tends to infinity gives a contradiction with \eqref{ineq_Kh_phi0}. Thus, $(K_h(u^n_h))_n$ is unbounded and the proof is completed.
\end{proof}
\begin{lem}\label{rmq_Kh_L1}
Let $0\leq k \leq 7$ and let $(u_h^n,\phi_h^n)$ be the solution of \eqref{disct_for_v}. Then $(u^n_h)_n$ blows up.
\end{lem}
\begin{proof}
If $0\leq k \leq 7$, then $\alpha_j:=\int_{-1}^1\varphi_j\, dx > 0$ for all $1\leq j \leq k+1$ (see table \ref{tab_poly}). Consequently, one may deduce from Proposition \ref{prop_stability} that if the initial data $(u_0,\phi_0)$ are positives, then $K_h(u_h^n) = \|u_h^n\|_1$ for $h$ small enough, where
$$\|u_h\|_1 = \dfrac{1}{b-a}\sum_{i=1}^I \frac{h}{2}\sum_{j=1}^{k+1}\alpha_j |u^i_j|.$$  It follows that $\|u_h^n\|_1\xrightarrow[n \to \infty]{}\infty$ and thus $\|u_h^n\|_{L^\infty(a,b)}\xrightarrow[n \to \infty]{}\infty$.
\end{proof}

Define 
\begin{align*}
G(z) = \sqrt{\frac{\lambda }{p+1}z^{p+1} + \gamma_h},
\end{align*}
then $G$ is a strictly increasing function in $[\alpha_h,\infty)$. In view of Proposition \ref{K_h}, we can proceed the same as in \cite{Cho} to prove the following.
\begin{lem}\label{lemcho}
There exists a constant $C>0$ independent of $h$ such that 
\begin{equation*}
T(h)\leq 2\Big( \int_{\alpha_h}^{\infty}\frac{dz}{G(z)} + Ch\Big). 
\end{equation*}
In particular, $(u_h^n)_n$ blows up in a finite time $T(h)$.
\end{lem}
\begin{proof}
See \cite[Lemma 5.3]{Cho}
\end{proof}
\begin{thm}\label{thm_of_blowup}
Let $(u,\phi)$ and $(u_h,\phi_h)$ be the solutions of \eqref{eq2} and \eqref{disct_for_v} respectively. Assume that $u_0 > 0$ and $u_1>0$ are large enough and $\phi_0 >0$. If $u_h$ weakly converges towards $u$, then $u^n_h$ blows up in finite time $T(h)$ and 
\begin{align}\label{T_h}
\lim_{h\to 0} T(h) = T_\infty.
\end{align}
\end{thm}
\begin{proof} We follow the strategy of \cite{Sasaki}. According to Lemma \ref{lemcho}, $u^n_h$ blows up in finite time $T(h)$. To establish \eqref{T_h}, we will prove the following inequalities:
\begin{equation}\label{T_inf}
T_\infty\leq \liminf_{h\to 0}T(h) = T_*,
\end{equation}
\begin{equation}\label{T_sup}
T_\infty \geq \limsup_{h\to 0} T(h) = T^*.
\end{equation}
Suppose that $T_*<T_\infty$ and let $\varepsilon = \frac{T_\infty-T_*}{2} > 0$. Then there exists $h_\varepsilon>0$ sufficiently small such that
$$T(h_\varepsilon)\leq T_* +\varepsilon < T_\infty.$$
On one hand, we have $\sup_{0\leq t\leq T_* +\varepsilon}\Vert u(\cdot,t)\Vert_{L^\infty(a,b)} < \infty$ and hence $$K_1 := \sup_{0\leq t\leq T_* +\varepsilon} K(u(t)) < \infty.$$
On the other hand, if $u_h^n \xrightharpoonup[h\to 0]{} u(t^n)$ then $K_h(u_h^n)\xrightarrow[h\to 0]{}  K(u(t^n))$. Hence, if $h_\varepsilon$ is sufficiently small, then $K_{h_\varepsilon}(u_{h_\varepsilon}^n) \leq K(u(t^n)) + \varepsilon$ for all $n$ such that $t^n < T_\infty$. It follows
\begin{align*}
\lim_{n\to \infty}K_{h_\varepsilon}(u_{h_\varepsilon}^n) & = \lim_{t^n\to T(h_\varepsilon)}K_{h_\varepsilon}(u_{h_\varepsilon}^n)\\
& \leq \lim_{t^n\to T(h_\varepsilon)} K(u(t^n)) + \varepsilon \\
& \leq K_1+\varepsilon,
\end{align*}
which contradicts Proposition \ref{K_h}, and hence \eqref{T_inf} holds. Next, suppose that $ T^* > T_\infty$ and let $N>0$ be the number of iterations to reach the time $T_\infty$, i.e. $T_\infty = t^N = \sum_{n=0}^{N-1} \Delta t^n$. Let $h_1 = \min(h_N,h_{N,\Lambda})$ with $h_N$ given in Proposition \ref{prop_stability} and $h_{N,\Lambda}$ given in Theorem \ref{thm_stab}, with $\Lambda_\infty = \|u_h^{0}\|_\infty + \|\phi_h^{0}\|_\infty$. Then $\forall\ h\in (0,h_1]$ and $0\leq n \leq N$ we have
\begin{align*}
\left|\sum_{i=1}^I\int_{K_i} u_h^{i,n}(x)\,dx\right| & \leq \sum_{i=1}^I\sum_{j=1}^{k+1} |u_j^{i,n}|\left|\int_{K_i}\varphi^i_j(x)\,dx\right| \\
& \leq \sum_{i=1}^I\frac{h}{2} \|u_h^{n}\|_\infty \sum_{j=1}^{k+1}\left|\int_{-1}^1\varphi_j(x)\,dx\right| \\
& \leq K_2\Lambda_\infty.
\end{align*}
% with $K_2 = b-a$ ${\big(}$\footnote{Since $0\leq k \leq 7$, then all the integrals are positives (see table \ref{tab_poly}). The result is then a direct conseqence of the Lagrange polynomials property $\sum_{j=1}^{k+1}\varphi_j(x)=1$ for any $x\in [-1,1]$.}${\big)}$. 
Let $\varepsilon = \frac{T^* - T_\infty}{4}$. Using Lemma \ref{lemcho}, there exist $h_\varepsilon>0$ (which will be fixed later) and $R\geq \frac{1}{b-a}\left(\|u\|_{L^\infty([a,b]\times [0,t^{N-1}])} + K_2\Lambda_\infty\right)$ such that
\begin{equation}\label{h**<eps}
\int_{R}^\infty \frac{dz}{G(z)}+C h_\varepsilon <\frac{\varepsilon}{2}.
\end{equation} 
It is shown in \cite{Gla73} that if the initial conditions are sufficiently large, then the solution $u$ of \eqref{eq2} blows up in $L^p$ norms, for any $1\leq p \leq \infty$ (see \cite[Theorem 2.1 and its Corollary]{Gla73}. We deduce that if the initial conditions are large enough, then there exists $t' = t'_R<T_\infty$ such that 
% From the hypothesis, there exists $t' = t'_R<T_\infty$ such that 
\begin{equation}\label{Ku>2R}
K(u(t))\geq 2R\quad \forall\, t\in [t',T_\infty).
\end{equation}
Set 
\begin{equation*}
T = t' + \frac{T_\infty - t'}{2} = \frac{t' + T_\infty}{2}<T_\infty,
\end{equation*}  
\begin{equation*}
h_* = \min\left\{ h_1,\, \left(\frac{T_\infty - t'}{2}\right)^\frac{1}{1+\sigma}\right\} %,\,\frac{R}{2\tilde{C}}
\end{equation*} 
% where $\tilde{C}$ is the constant given by the estimate \eqref{u-pu}, 
and let $h\in(0,h_*]$. Then we have for all $n\geq 0$ such that $t^n < T_\infty$
\begin{align*}
\vert K(u(t^n)) - K_h(u^n_h)\vert & = \frac{1}{b-a}\left|\sum_{i=1}^I\int_{K_i} u(x,t^n) - u^{i,n}_h(x) dx \right|\\
&\leq \frac{1}{b-a}\left( \|u(t^n)\|_{L^\infty([a,b])} +  \left|\sum_{i=1}^I\int_{K_i} u^{i,n}_h(x)dx\right| \right).
\end{align*}
In particular, we obtain for all $0\leq n \leq N-1$ %Using \eqref{u-pu} and assumption $c)$, we deduce $\forall\, n\geq n_0$ such that $t^n < T_\infty$
$$ \vert K(u(t^n)) - K_h(u^n_h)\vert \leq \frac{1}{b-a}\left( \|u\|_{L^\infty([a,b]\times [0,t^{N-1}])} + K_2\Lambda_\infty \right) \leq R.$$
It follows
\begin{align*}
K_h(u^n_h)\geq K(u(t^n)) - R \quad \forall\, 0\leq n\leq N-1.
\end{align*}
Recall that $\Delta t^n \leq h^{1+\sigma} \leq T-t'<T_\infty-t'$. Since $T^*>T_\infty$, then there exists $n_1\leq N-1$ such that $t'\leq t^{n_1}<T_\infty$. We deduce from \eqref{Ku>2R}
\begin{align}\label{diff_k_R}
K_h(u^{n_1}_h)\geq K(u(t^{n_1}))-R \geq R.
\end{align}
Now, using $\dis{\limsup_{h\to 0}T(h)} = T^* > T_\infty$, one may choose $h_\varepsilon \leq h_*$ sufficiently small such that 
\begin{align*}
T(h_\varepsilon) \geq T_\infty +\varepsilon.
\end{align*}
However, in view of Lemma \ref{lemcho} and equations \eqref{diff_k_R} and \eqref{h**<eps}, we have 
\begin{align*}
T(h_\varepsilon) = t^{n_1} + \sum_{n=n_1}^\infty \Delta t^n & < T_\infty + 2\left( \int_{K_{h_\varepsilon}(u^{n_1}_{h_\varepsilon})}^\infty \frac{dz}{G(z)} + Ch_\varepsilon\right)  \\
& \leq T_\infty + 2\left( \int_R^\infty \frac{dz}{G(z)} + C h_\varepsilon\right) \\
& < T_\infty + \varepsilon,
\end{align*}
which is a contradiction. This achieves the proof.
\end{proof}
%
%%%%%%%%%%%%%%%%%%%%%%%%%%%%%%%%%%%%%%%%%%%%%%%%%%%%%%%%%%%%%%%%%%%%%%%%%%%%%%%%%%%%%%%%%%%%%
%
\section{Numerical examples}\label{Sec_exmp}
In this section, we present some numerical tests in order to illustrate our method.  For all the examples, we consider the DG scheme \eqref{disct_scheme} with $\IP_1$ approximation. The simulations have been performed using the software Matlab.%\footnote{https://www.mathworks.com/products/matlab.html}.

\begin{exmp}\label{exemple1}
In this example, we consider constant initial conditions so that the solution of \eqref{eqq1} is space independent. The exact solution we consider is $$u(t) = \mu (T-t)^\frac{2}{1-p} $$ with $\mu = \left(2\frac{p+1}{(p-1)^2}\right)^\frac{1}{p-1}$. We perform two test cases with $p=2$ and $p=3$. The blow-up time for both cases is set to $T=0.1\, s$. Figure \ref{fig_edo} shows a comparison between the exact solution and the numerical solution functions of the time. One can notice a very good superposition of the solutions (with relative errors less than $1\%$ in both $L^2$ and $L^\infty$ norms), which justifies the validity of the explicit Euler scheme as an appropriate choice for the time discretization of the DG method.

\begin{figure}
\begin{center}
\includegraphics[width=5.5cm,height=5cm]{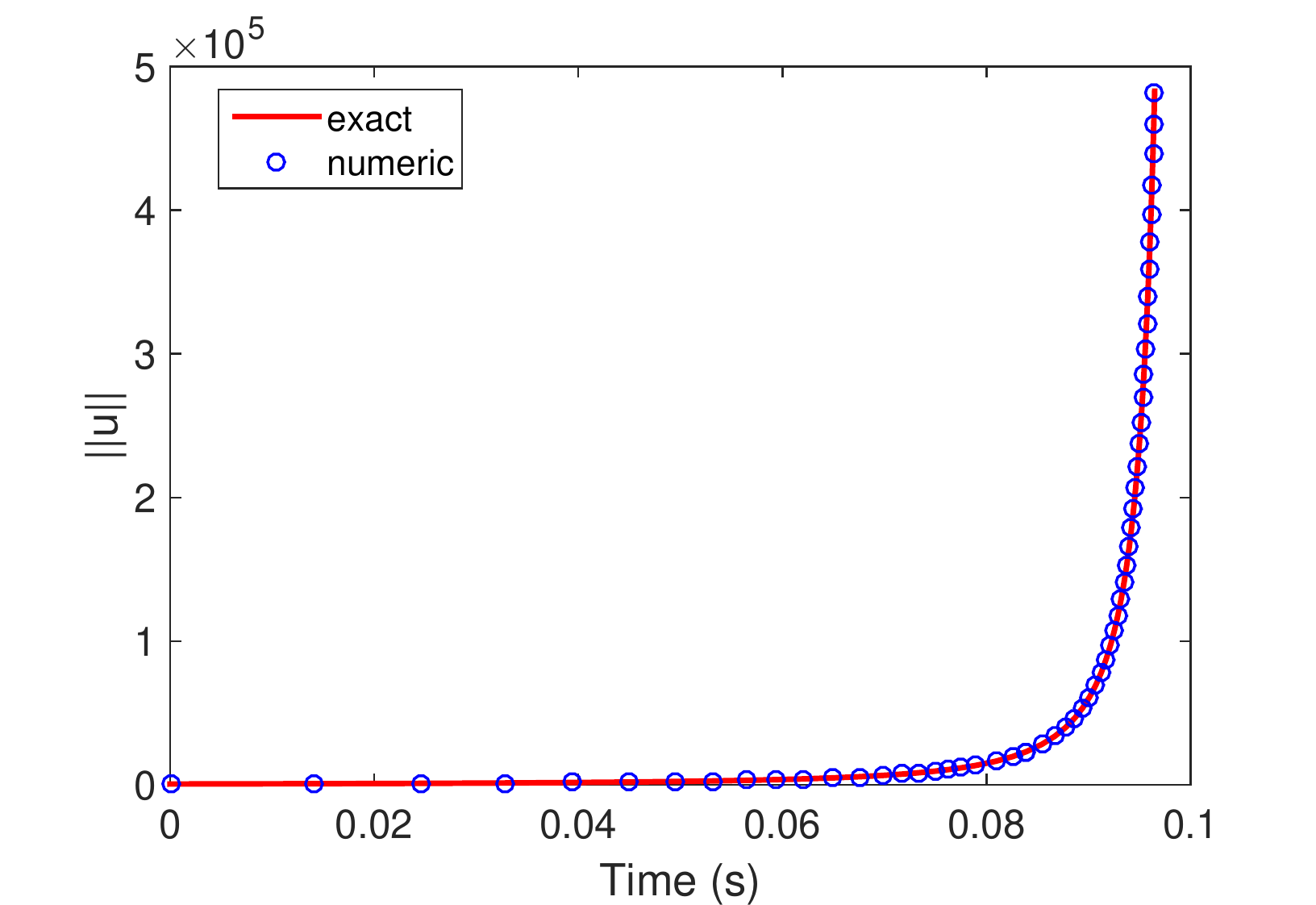} \includegraphics[width=5.5cm,height=5cm]{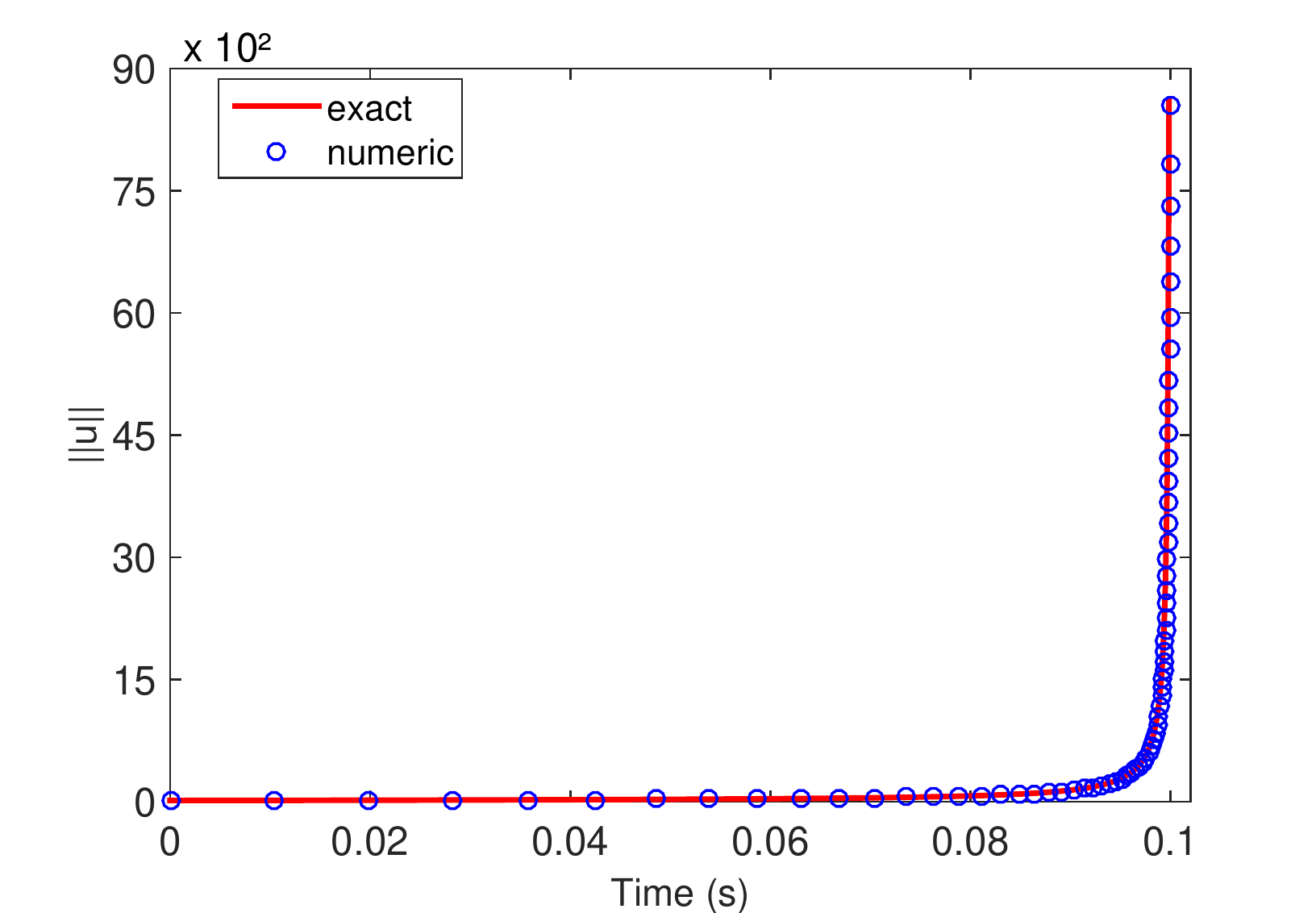}
\caption{\label{fig_edo} Comparison between the numerical solution (blue circles) and the exact solution (red line) for p = 2 (left) and p=3 (right).}
\end{center}
\end{figure}
\end{exmp}

\begin{exmp}\label{exemple1bis}
We consider an exact solution of \eqref{eqq1} given by 
\begin{equation}\label{sol_ana}
u(x,t) = \mu (T-t+d\,x)^\frac{2}{1-p}
\end{equation}
with $\mu = \left(2(1-d^2)\frac{p+1}{(p-1)^2}\right)^\frac{1}{p-1}$ and $d \in (0,1)$ is an arbitrary parameter. Figures \ref{fig_edp1} and \ref{fig_edp2} show a comparison between the exact solution and the numerical solution at various times, for $p=2$ and $p=3$ respectively. The parameters used are $T=0.5\,s$ and $d=0.01$. One can notice that the numerical solutions fit very well with the exact solutions at all the recorded times. The relative errors in $L^\infty$ norms is less than $1\%$ if a refined mesh is used. We also investigate the blow-up curve in the following way. Let $R\geq \min_{x\in [0,1]} u(x,0) = \mu (T+d)^\frac{2}{1-p}$, and let $\xi_R$ the function defined by $u(x,\xi_R(x))=R$. It is easy to show from \eqref{sol_ana} that $\xi_R$ is a straight line given by $\xi_R(x) = T-\left(\frac{\mu}{R}\right)^{\frac{p-1}{2}} + d\,x$. When $R$ goes to infinity, $\xi_R(x)$ tends to the blow-up time $T_\infty(x) = T+d\,x$, for any $x\in [0,1]$. Thus, one can approximate numerically the blow-up curve $T_\infty$ by computing $\xi_R$ for large values of $R$. %Since equalities do not occur for numerical values in general, 
 In practice, we define $\xi_R$ as $$\xi_R(x)= \inf\{t\geq 0,\ |u(x,t)|\geq R\}.$$
Figure \ref{fig_BLup} shows $\xi_R$ function of $x$ for various values of $R$. We notice that  $\xi_R$ is a straight line with slope equal to $d$ for all values of $R$, which is in accordance with the theory. Furthermore, as the parameter $R$ gets bigger, one can notice that $\xi_R$ gets closer to the theoretical blow-up curve $T_\infty$. % This shows the efficiency of our proposed DG method.

\begin{figure}
\begin{center}
\includegraphics[width=5.5cm,height=5cm]{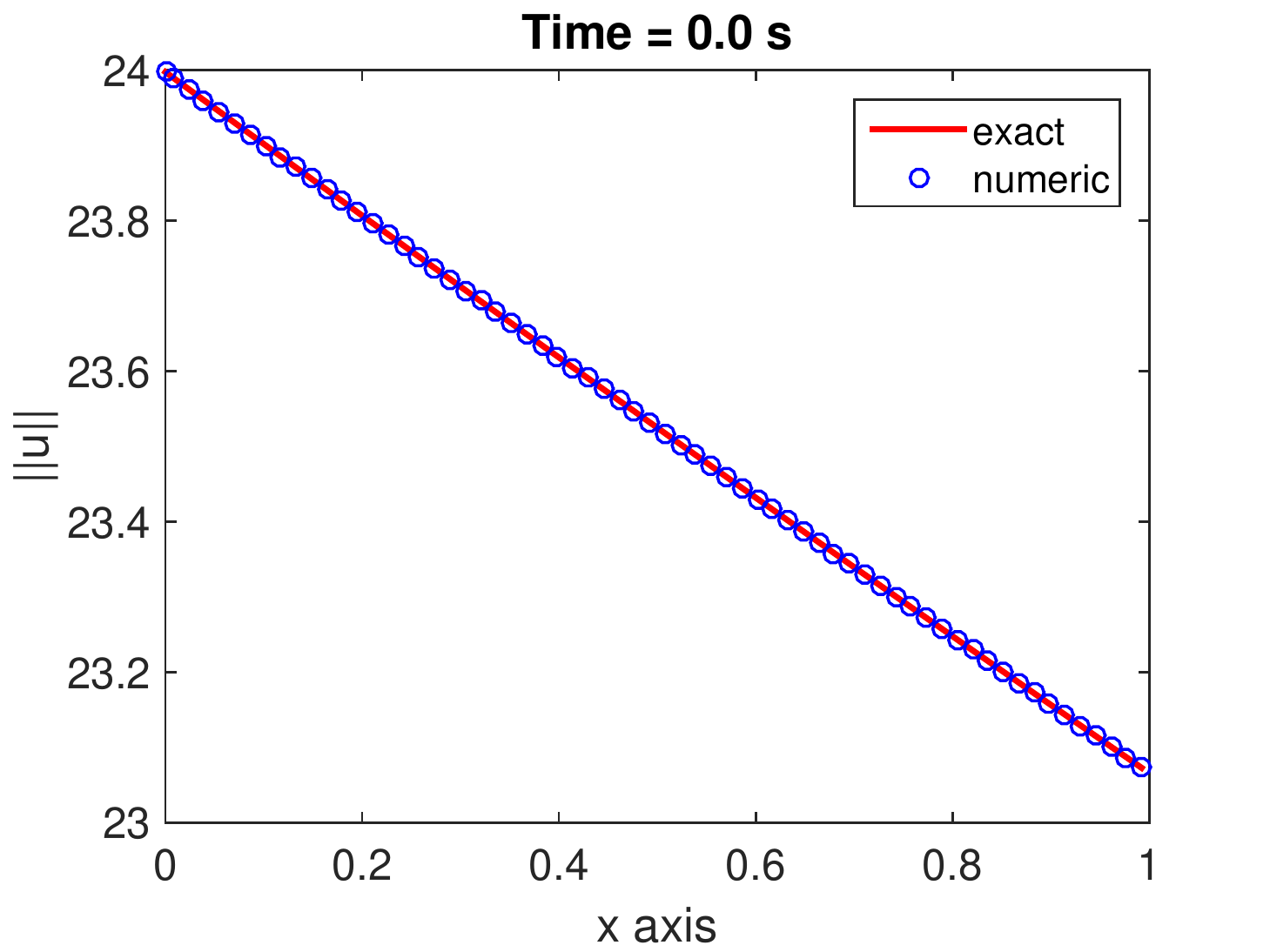} \includegraphics[width=5.5cm,height=5cm]{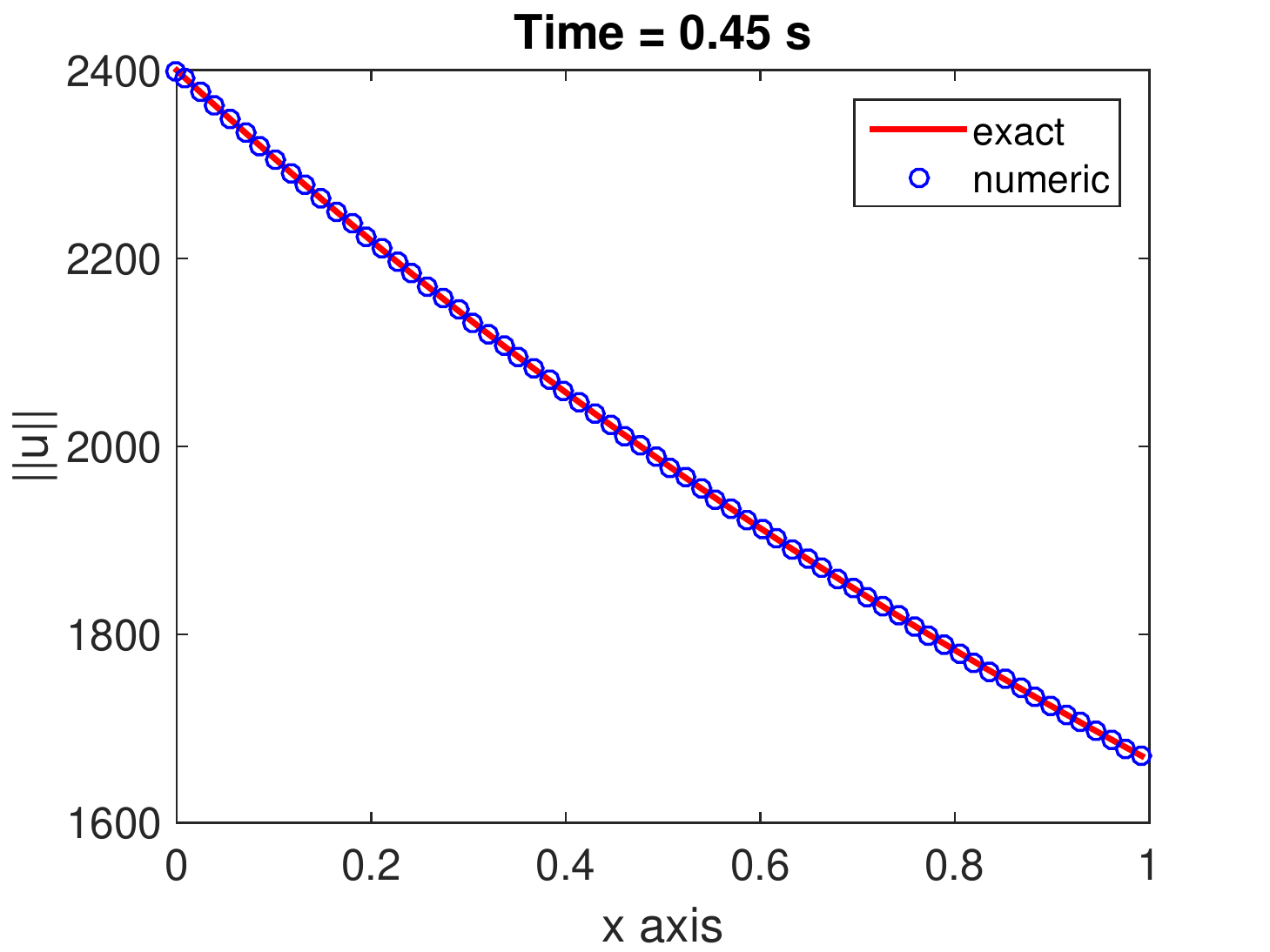}
\includegraphics[width=5.5cm,height=5cm]{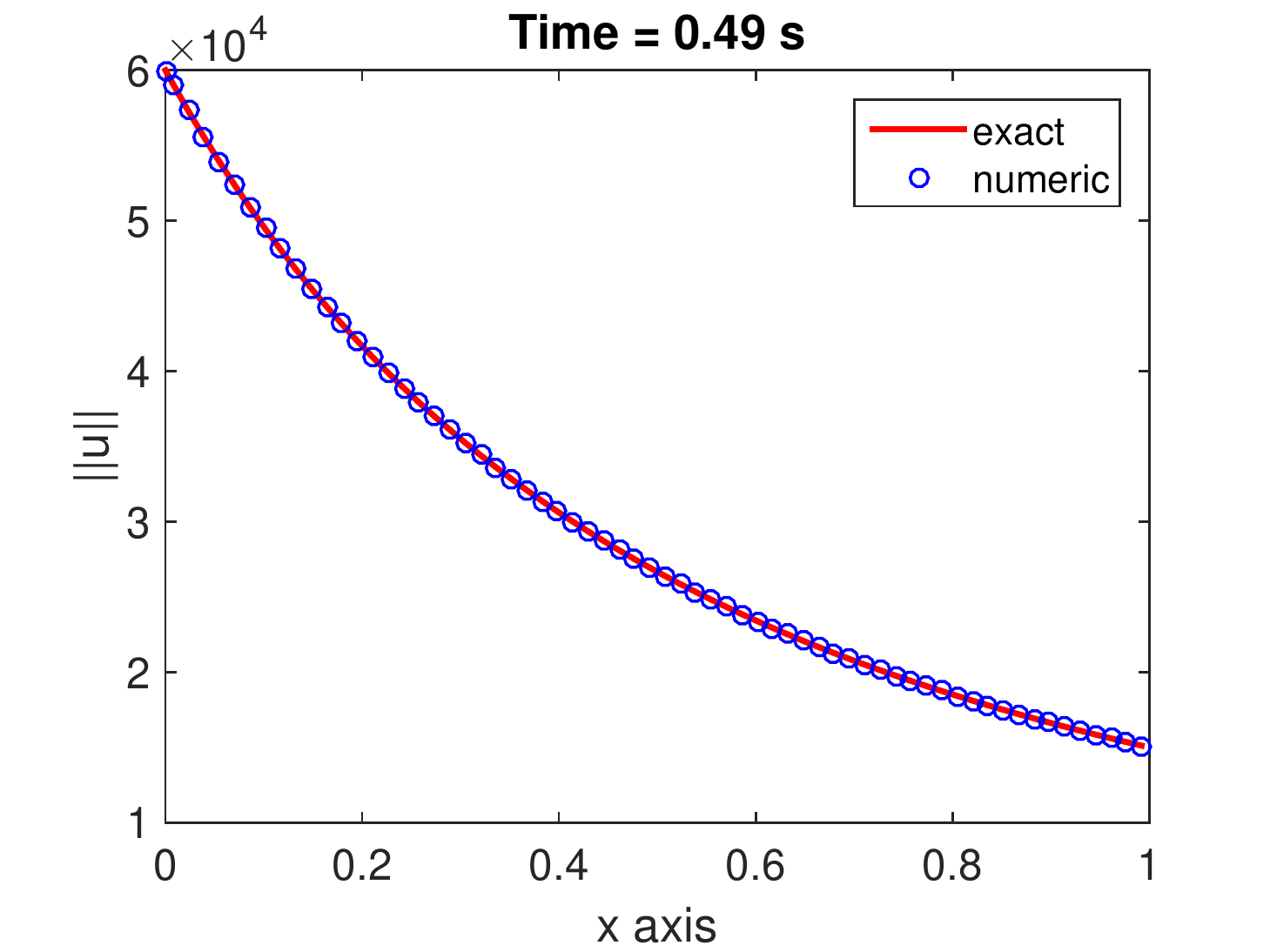} \includegraphics[width=5.5cm,height=5cm]{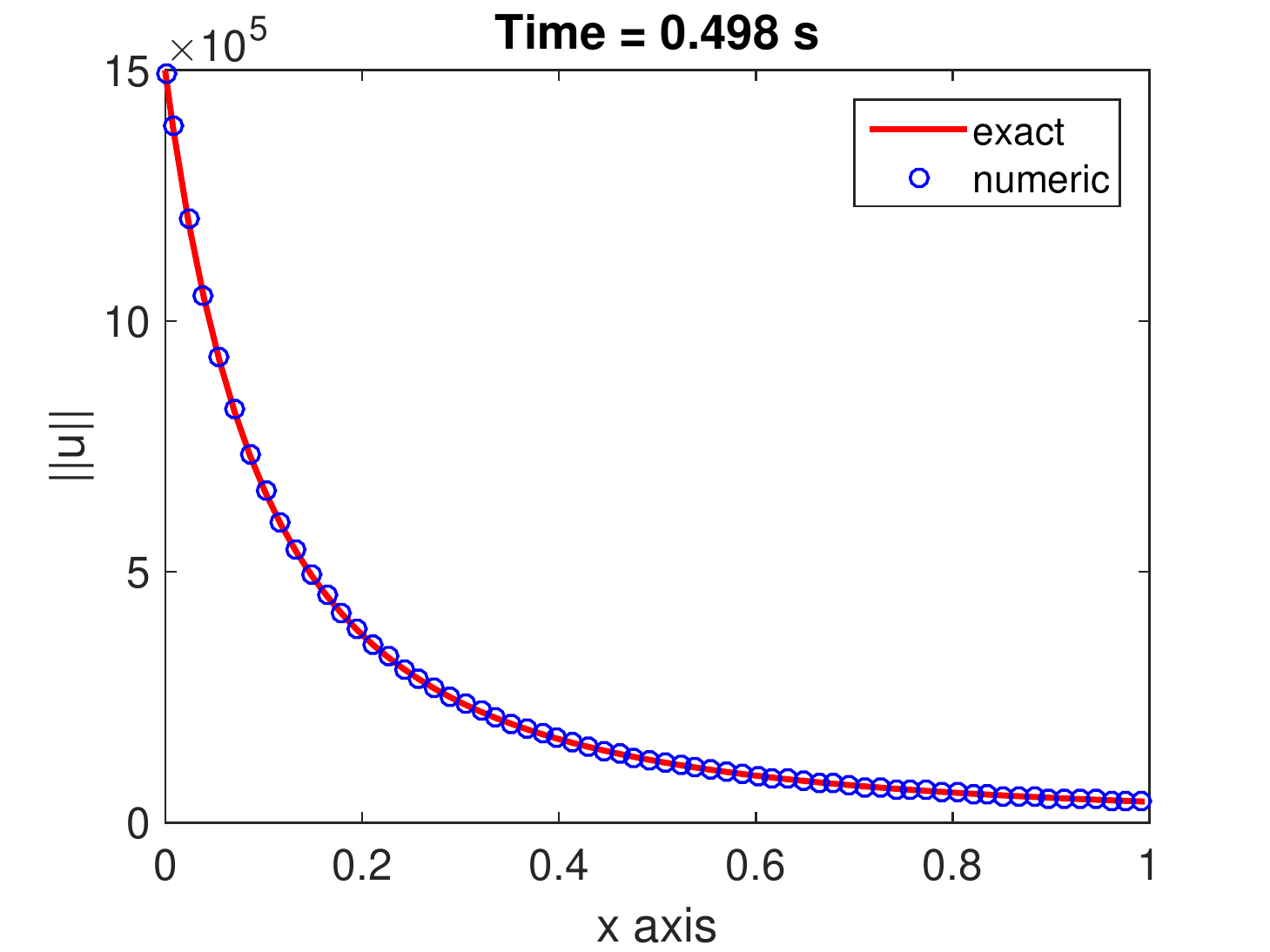}
\caption{\label{fig_edp1} Comparison between the numerical solution (blue circles) and the exact solution (red line) at various times. Case $p=2$.}
\end{center}
\end{figure}

\begin{figure}
\begin{center}
\includegraphics[width=5.5cm,height=5.cm]{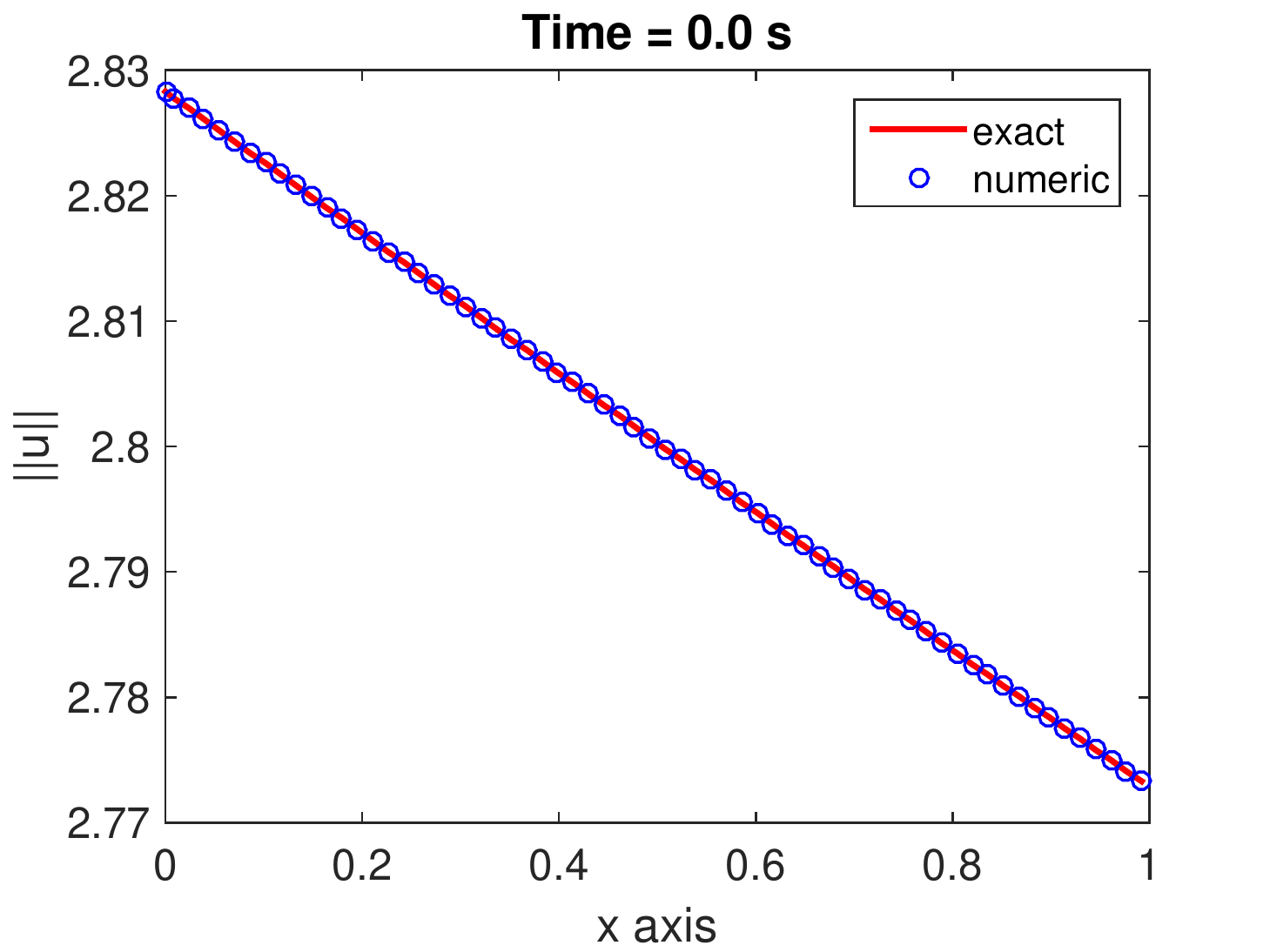} \includegraphics[width=5.5cm,height=5.cm]{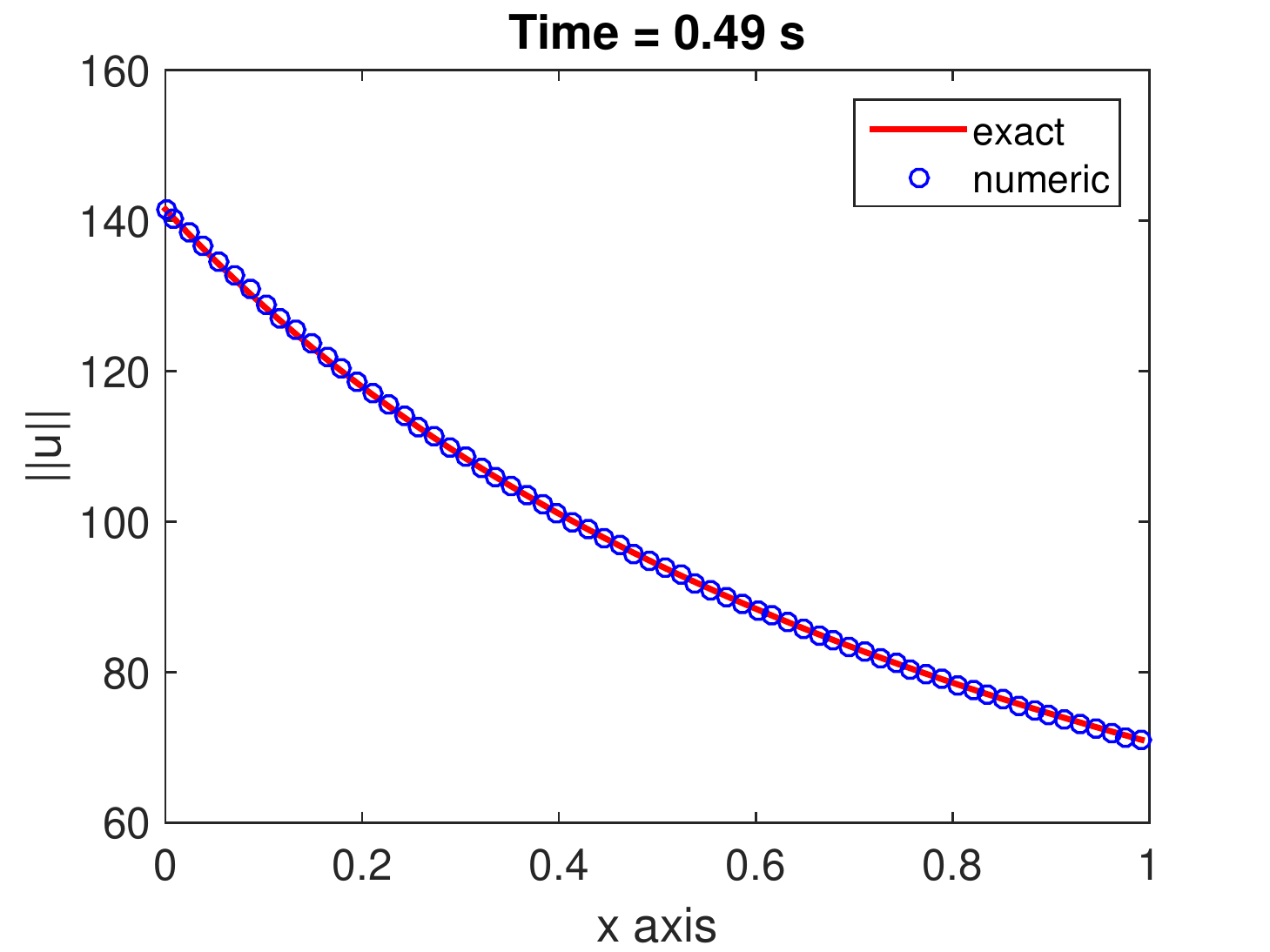} %\newline
\includegraphics[width=5.5cm,height=5.cm]{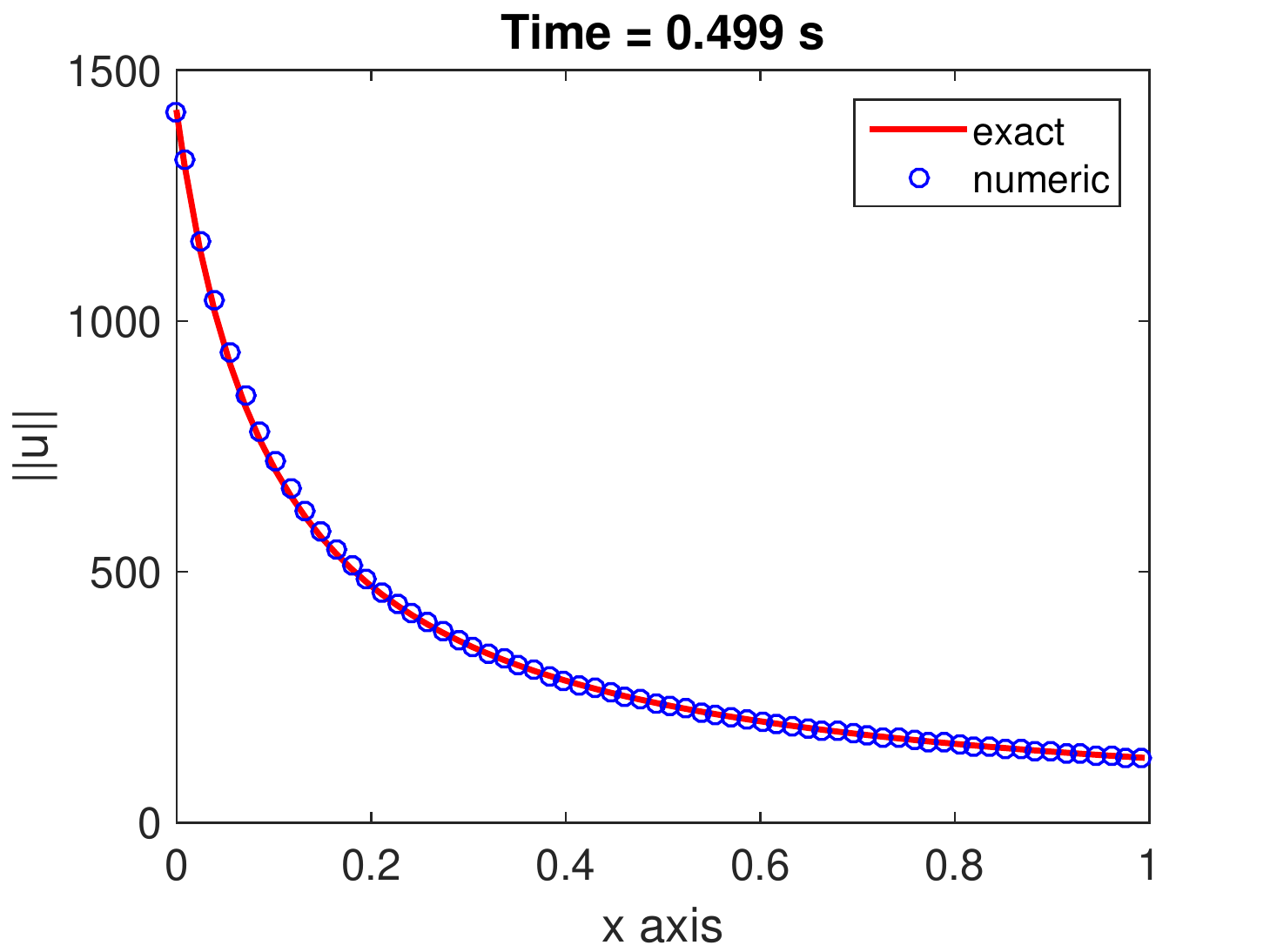} \includegraphics[width=5.5cm,height=5.cm]{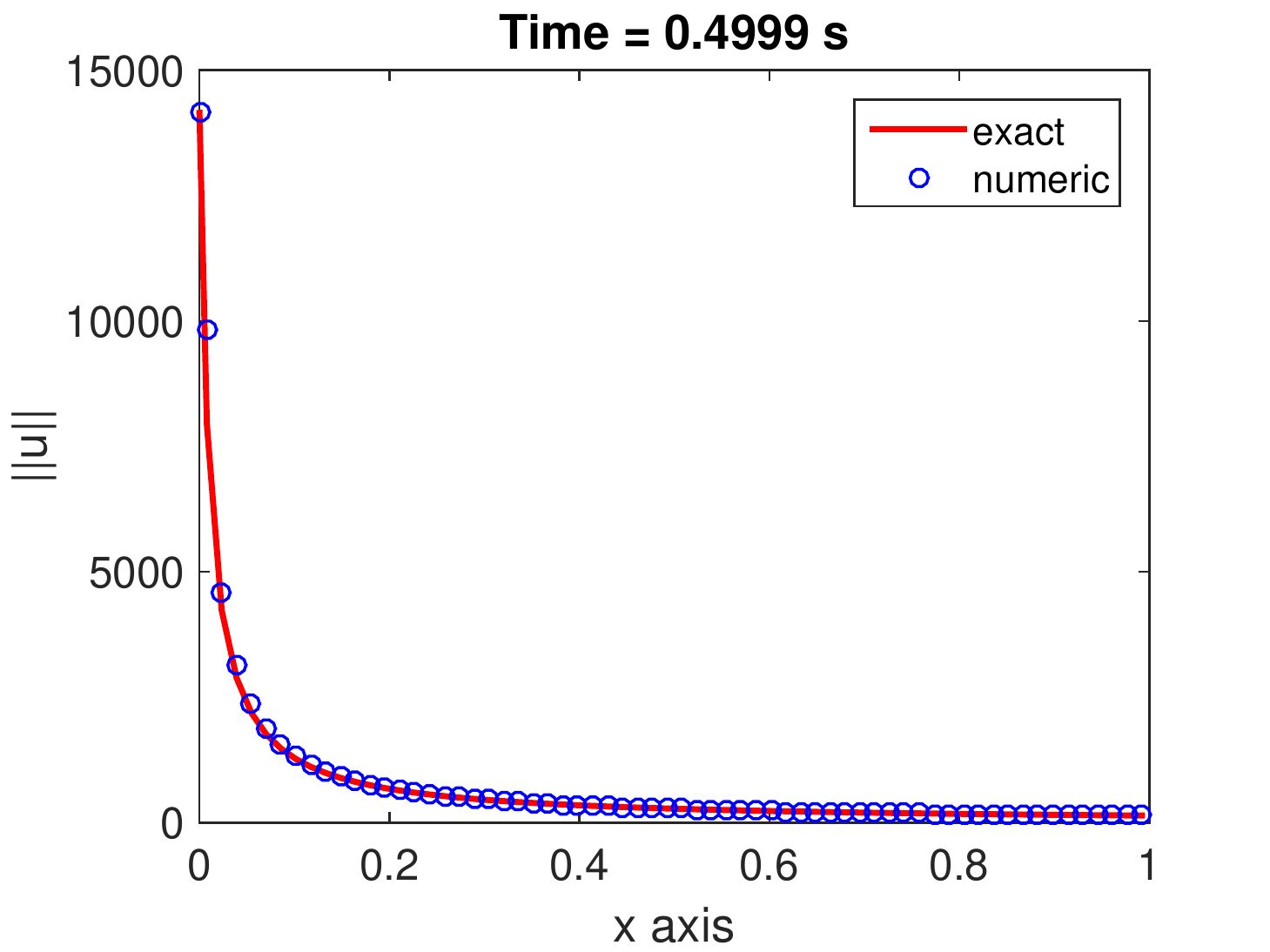}
\caption{\label{fig_edp2} Comparison between the numerical solution (blue circles) and the exact solution (red line) at various times. Case $p=3$.}
\end{center}
\end{figure}

\begin{figure}
\begin{center}
\includegraphics[width=7cm,height=6cm]{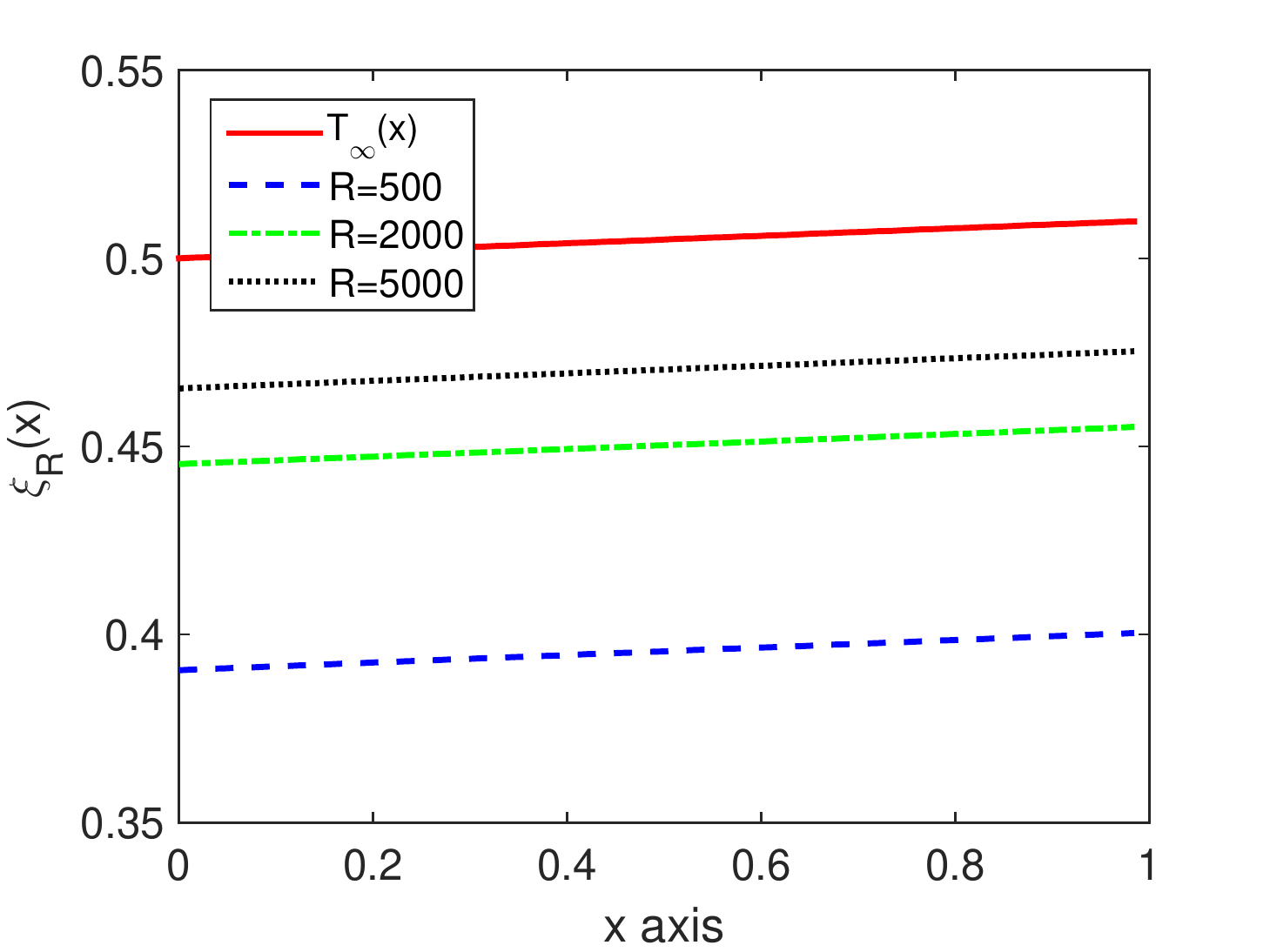}\includegraphics[width=7cm,height=6cm]{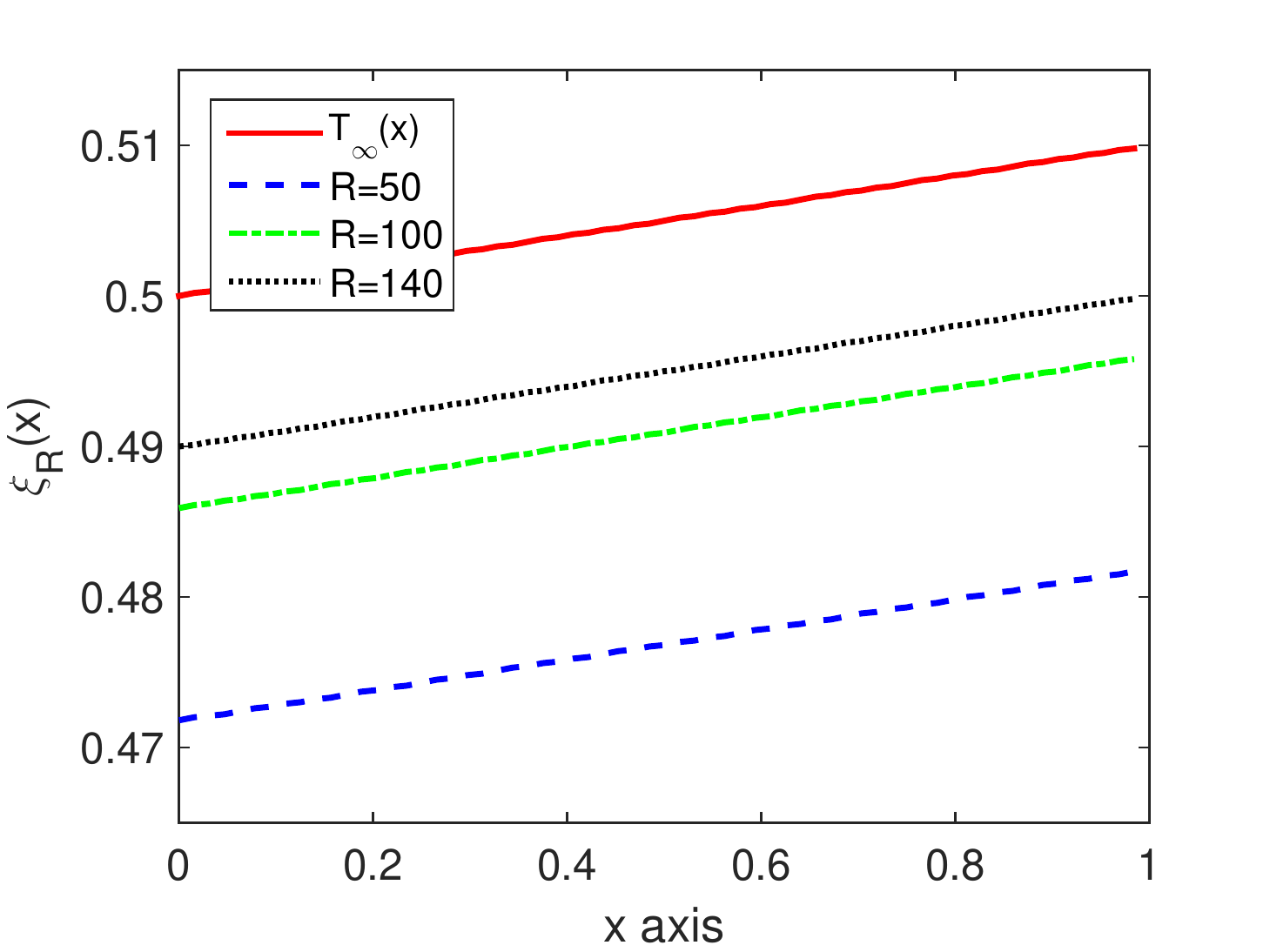}
\caption{\label{fig_BLup} $\xi_R$ for various values of $R$. Case $p=2$ (left) and $p=3$ (right). As $R$ increases, $\xi_R$ converges (pointwise and uniformly) toward the blow-up curve (red line).}
\end{center}
\end{figure}
\end{exmp}

\begin{exmp}\label{exemple2}
In this example, we consider the system \eqref{eqq1} with initial conditions 
\begin{align*}
u_0(x) & = 5(\sin(4\pi x) + 2),\\
u_1(x) & = 5(\sin(4\pi x) -4\pi \cos(4\pi x) + 2).
\end{align*}
With such conditions, we have $u_0=\phi_0>0$, $\alpha=K(u_0)=10>0$ and $\beta=K(u_1)=10>0$, so that all the hypothesis of  Proposition \ref{prop_stability} and Proposition \ref{K_h} are satisfied. Accordingly, we expect the solution to blow up in a finite time. Figure \ref{fig_exmp2_1} and figure \ref{fig_exmp2_2} show the evolution of the numerical solutions in space-time axes for $p=2$ and $p=3$ respectively.
\end{exmp}

\begin{figure}
\begin{center}
\includegraphics[width=10cm,height=8cm]{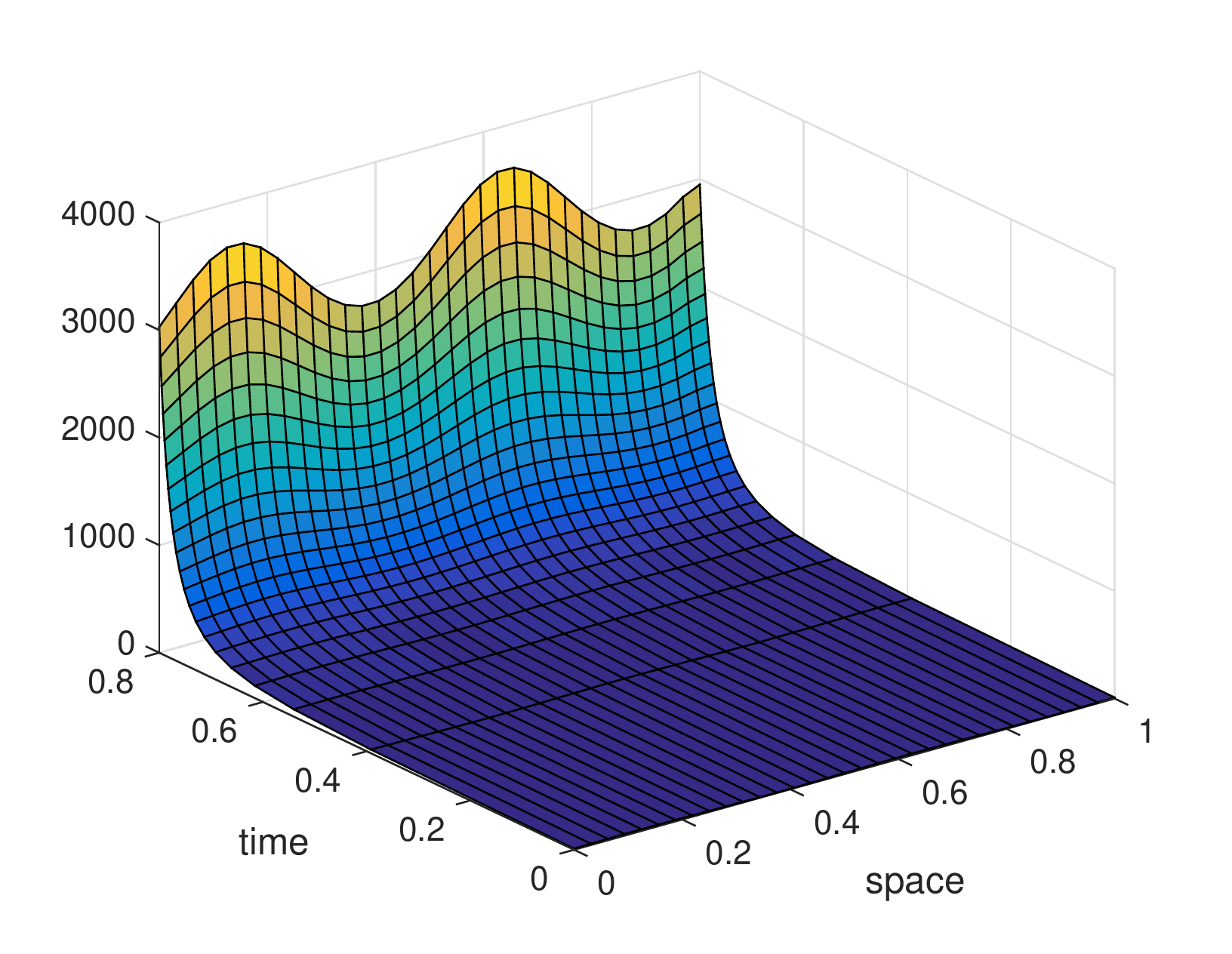}
\caption{\label{fig_exmp2_1} Numerical solution of example \ref{exemple2} with $p = 2$.}
\end{center}
\end{figure}

\begin{figure}
\begin{center}
\includegraphics[width=10cm,height=8cm]{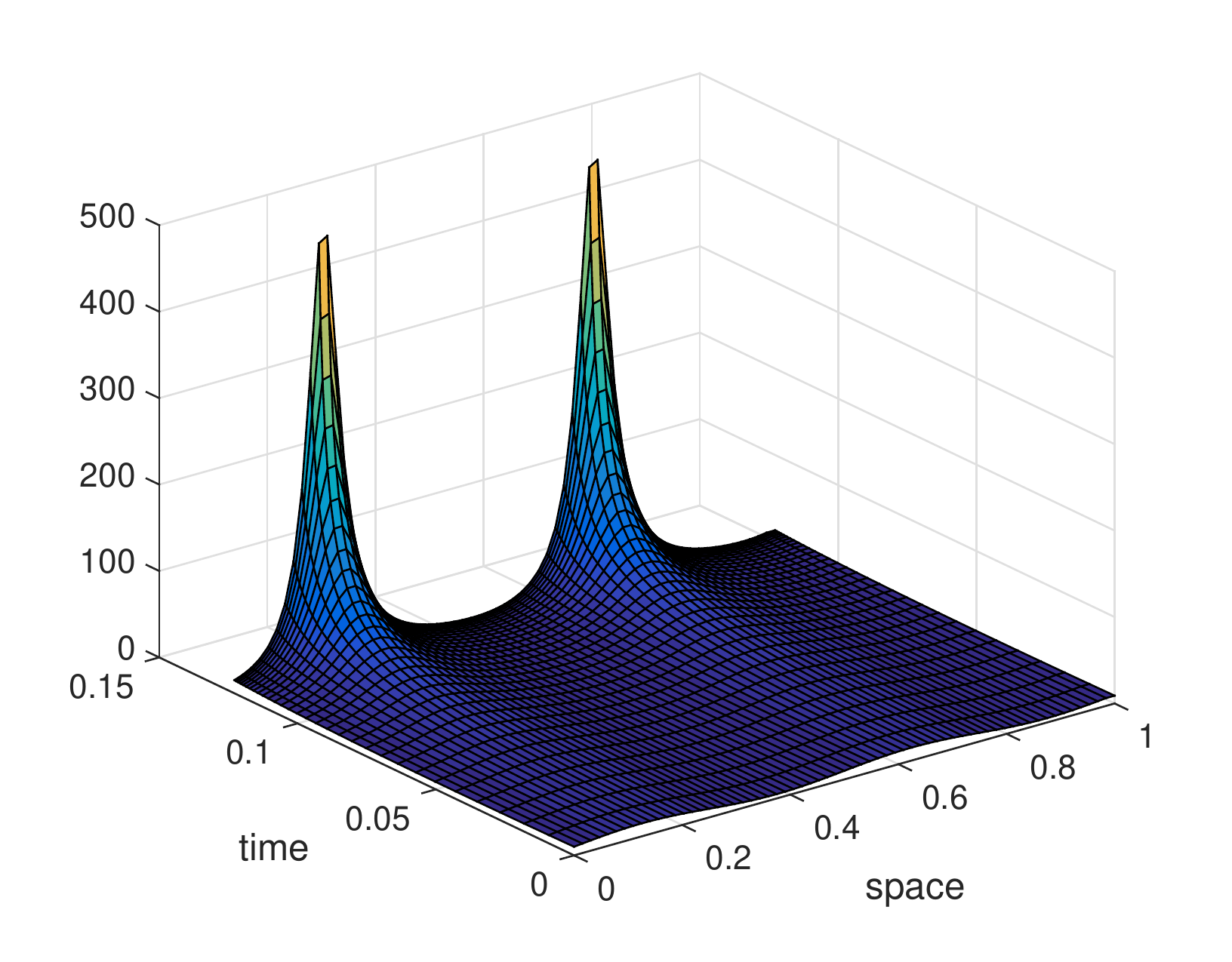}
\caption{\label{fig_exmp2_2} Numerical solution of example \ref{exemple2} with $p = 3$.}
\end{center}
\end{figure}

\begin{exmp}\label{exemple3}
In this example, we compare our DG method to a finite difference (FD) method developed in \cite{Sasaki}. Let us mention that the authors also proved that their FD scheme is convergent, as well as the numerical blow-up time, toward the exact solution. We use a very refined grid mesh for the FD algorithm in order to obtain results as accurate as possible\footnote{The FD grid is 16 times finer than the DG grid.}. The initial conditions used are $u_0(x) = 5(\sin(4\pi x) + 2)$ and $u_1(x) = 20\pi + 5$. Figure \ref{fig_exmp3} shows a comparison between the numerical solutions in various time for $p=2$ and $p=3$. One can notice a very good superposition between the solutions in all recorded times. In table \ref{tab_norm_dg_df}, we report the relative $L^2$ and $L^\infty$ errors between the FD and the DG solutions at the different times. Moreover, we checked the convergence of the numerical blow-up time when the space path $h$ goes to zero. Table \ref{tab1} and figure \ref{fig_time} show the blow-up times of the DG method versus the FD method function of $h$ for $p=3$. Since the blow-up time can not be reached in finite steps (see Definition \ref{def_blow}), we fixed $\|u^n_h\|_\infty \geq 10^{9}$ as a threshold criterion in order to stop the iterations. One can notice that both the DG and the FD algorithms seem to converge toward the same limit, which is $T_\infty\simeq 1.14\,s$ in this case. This confirms the efficiency of our proposed method.
\end{exmp}

\renewcommand{\arraystretch}{1.25}
\begin{table}
\hspace*{-5mm}
\begin{minipage}{7.5cm}
\begin{tabular}{|*{3}{c|}}
\hline
\multicolumn{3}{|c|}{$p = 2$}\\ 
\hline 
Time (s) & \rule[12pt]{0pt}{5pt} $\frac{\|u_h^{FD}-u_h^{DG}\|_2}{\|u_h^{FD}\|_2}$ \rule[-8pt]{0pt}{5pt} & $\frac{\|u_h^{FD}-u_h^{DG}\|_\infty}{\|u_h^{FD}\|_\infty}$ \\ 
\hline 
$0.03$  & $2.16\times 10^{-3}$ & $1.95\times 10^{-3}$ \\ 
\hline 
$0.10$  & $9.15\times 10^{-4}$ & $9.32\times 10^{-4}$ \\ 
\hline 
$0.15$  & $5.97\times 10^{-4}$ & $9.34\times 10^{-4}$ \\ 
\hline 
$0.25$  & $1.11\times 10^{-3}$ & $1.66\times 10^{-3}$ \\ 
\hline 
\end{tabular}
\end{minipage}
\begin{minipage}{7.5cm}
\begin{tabular}{|*{3}{c|}}
\hline
\multicolumn{3}{|c|}{$p = 3$}\\ 
\hline 
Time (s) & \rule[12pt]{0pt}{5pt} $\frac{\|u_h^{FD}-u_h^{DG}\|_2}{\|u_h^{FD}\|_2}$ \rule[-8pt]{0pt}{5pt}& $\frac{\|u_h^{FD}-u_h^{DG}\|_\infty}{\|u_h^{FD}\|_\infty}$ \\ 
\hline 
$0.03$  & $2.58\times 10^{-4}$ & $3.46\times 10^{-4}$ \\ 
\hline 
$0.09$  & $1.25\times 10^{-3}$ & $1.95\times 10^{-3}$ \\ 
\hline 
$0.105$ & $4.05\times 10^{-3}$ & $5.96\times 10^{-3}$ \\ 
\hline 
$0.110$ & $9.98\times 10^{-3}$ & $1.39\times 10^{-2}$ \\ 
\hline 
\end{tabular}
\end{minipage}
\caption{\label{tab_norm_dg_df} Relative errors of the DG solutions $u_h^{DG}$ versus the FD solutions $u_h^{FD}$ for various times.}
\end{table}

\begin{figure}
\begin{center}
\includegraphics[width=6cm, height=6cm]{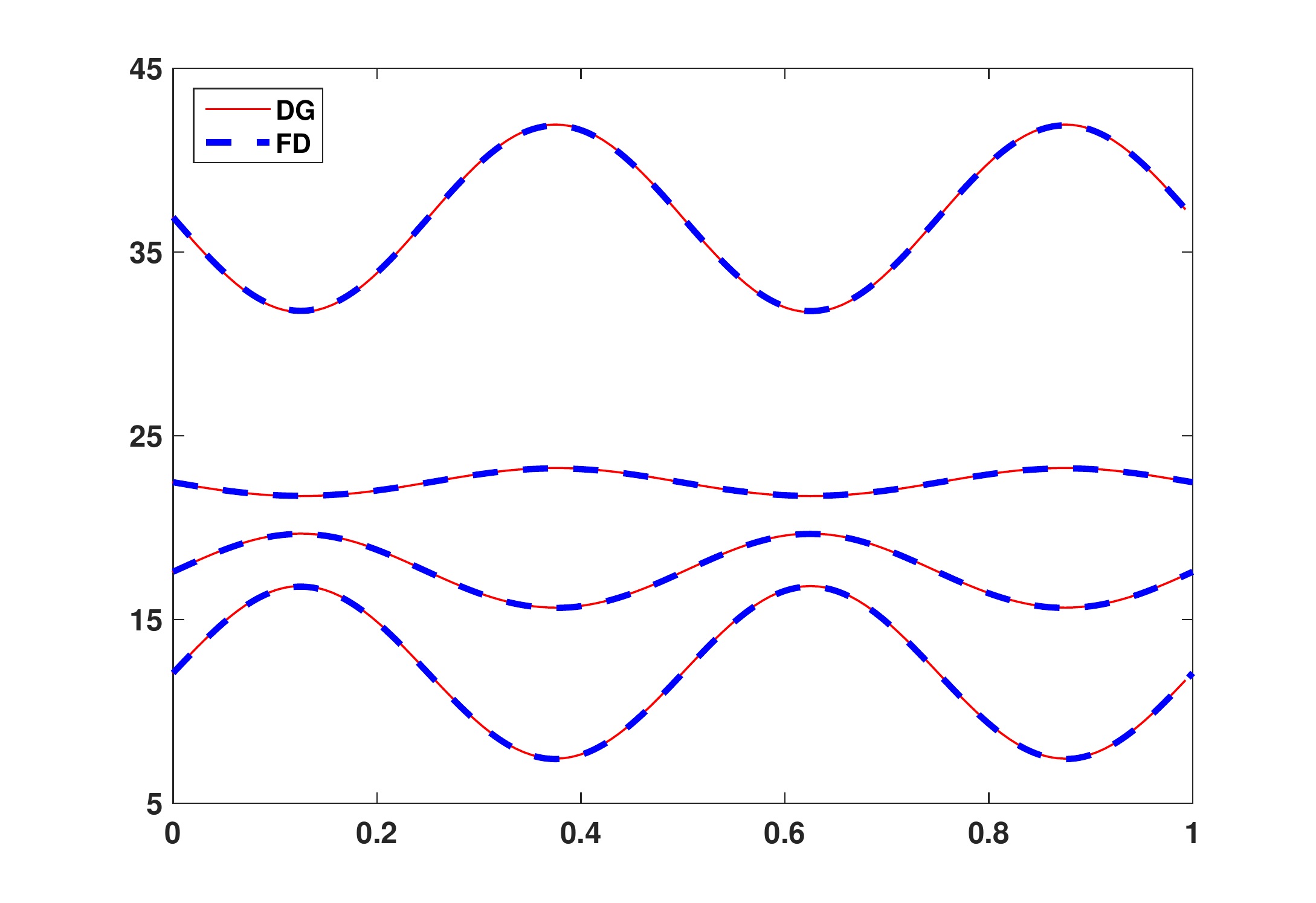}
\includegraphics[width=6cm, height=6cm]{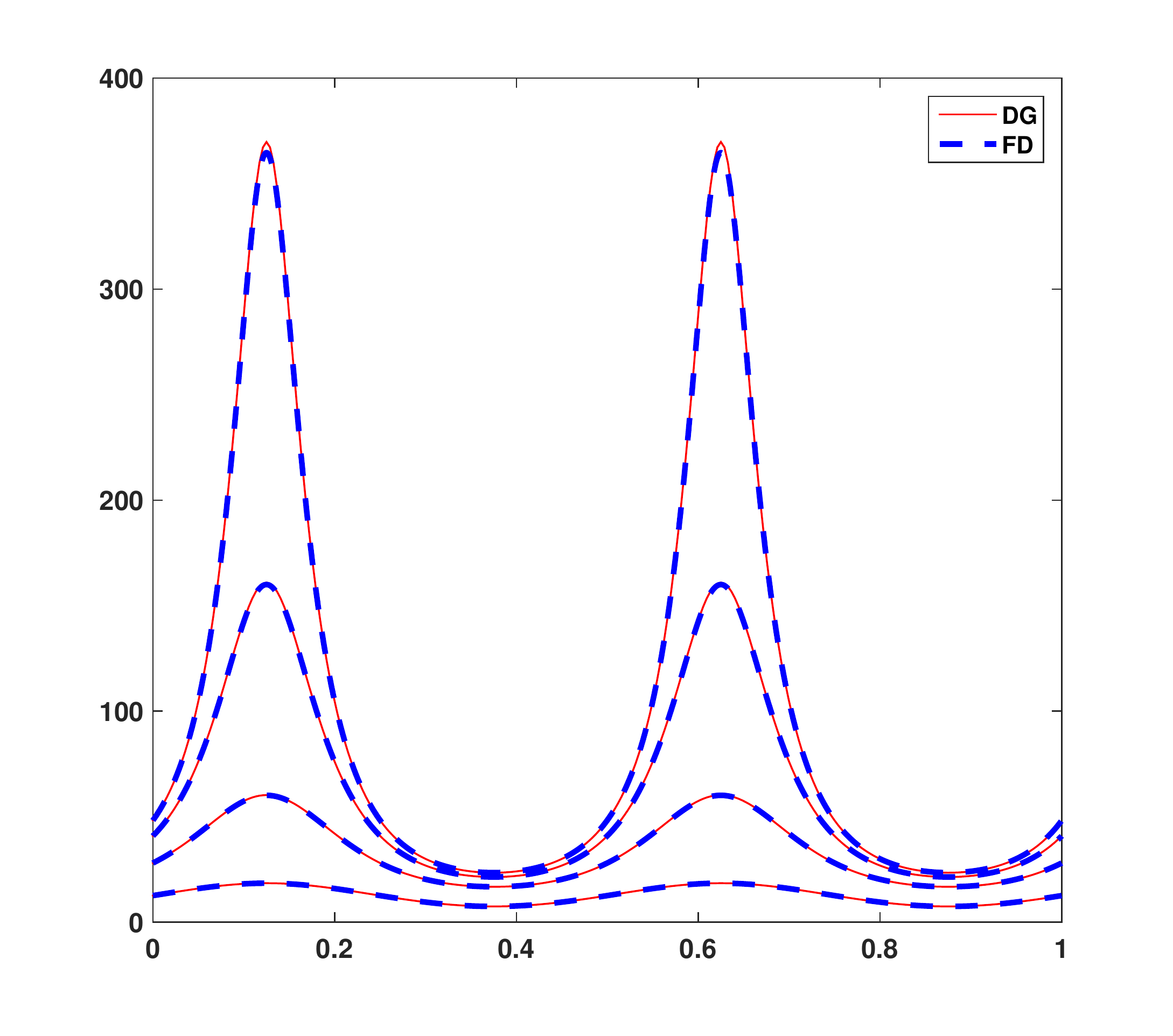}
\caption{\label{fig_exmp3} Comparison between the DG solution (red line) and the FD solution \cite{Sasaki} (blue dash) of example \ref{exemple3} at various times. Case $p=2$ (left) and $p = 3$ (right).}
\end{center}
\end{figure}

\renewcommand{\arraystretch}{1.25}
\begin{table}
\begin{center}
\begin{tabular}{|*{3}{c|}}
\cline{2-3} 
\multicolumn{1}{c|}{} & \multicolumn{2}{c|}{$T(h)$} \\ 
\hline 
$h$ & DG & FD \\ 
\hline 
$1/2^5$  & $1.1671$ & $1.1675$ \\ 
\hline 
$1/2^6$ & $1.1527$ & $1.1538$ \\ 
\hline 
$1/2^7$  & $1.1455$ & $1.1463$ \\ 
\hline 
$1/2^8$  & $1.1419$ & $1.1423$ \\ 
\hline 
$1/2^9$  & $1.1401$ & $1.1403$ \\ 
\hline 
\end{tabular}
\caption{\label{tab1} Blow-up time function of $h$. Case p=3.}
\end{center}
\end{table}

\begin{figure}
\begin{center}
\includegraphics[scale=0.4]{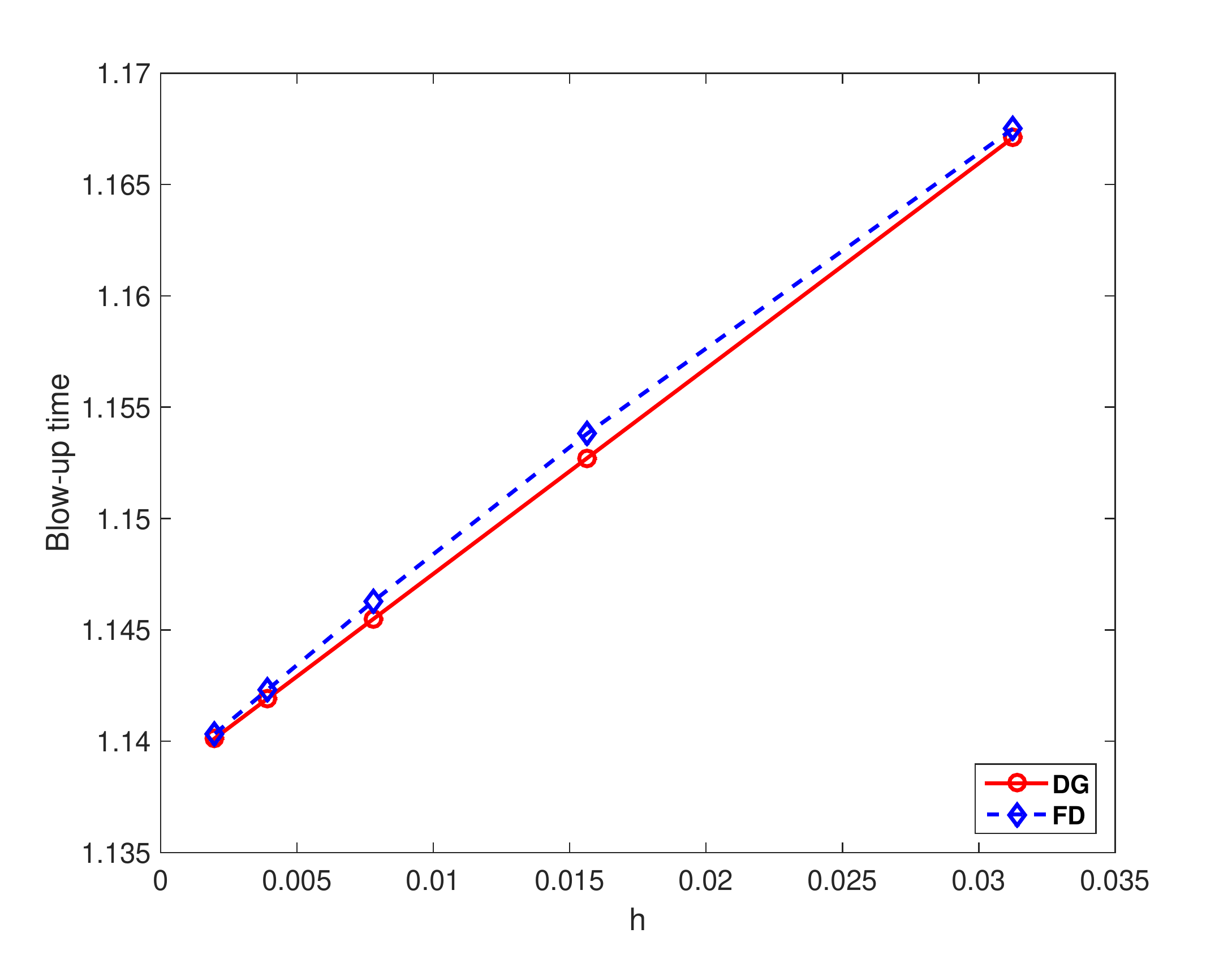}
\caption{\label{fig_time} Comparison between the numerical blow-up time of the DG method (red line) and the FD method (dashed blue line) of example \ref{exemple3}. Case $p = 3$.}
\end{center}
\end{figure}
%
%%%%%%%%%%%%%%%%%%%%%%%%%%%%%%%%%%%%%%%%%%%%%%%%%%%%%%%%%%%%%%%%%%%%%%%%%%%%%%%%%%%%%%%%%%%%%
%
\section{Conclusion}
In this paper, we developed a numerical scheme based on discontinuous Galerkin (DG) formulation for the approximation of the nonlinear wave equation in one dimensional space. We showed that the DG scheme is consistent, stable (in the sense that the numerical solution do not blows up in a finite number of iterations, i.e. before the exact blow-up time) and converges toward the exact solution. For the time update, we used an explicit Euler scheme. Since blow-up phenomena can occur, one may not expect a constant time increment\footnote{Otherwise, the numerical solution could be computed beyond the blow-up time leading to erroneous results. Actually, the author in \cite{Cho13} showed that a constant time step remains also applicable if an appropriate stopping criterion is specified.}. Instead, we used a refined time meshing, with time step inversely proportional to the solution's amplitude. Since we are dealing with transport equations, the CFL condition is more constrained in case of DG methods. Indeed, the classical theory of the DG methods shows that $\Delta t$ should be of order $(\Delta x)^{3/2}$ (rather than the standard $\Delta x$) to ensure the stability of the method \cite{Chavent89,Cockburn et Shu}\footnote{While the order $3/2$ has been theoretically established for the linear problems, it has been observed numerically that the order one, i.e. $\Delta t = O(\Delta x)$, is sufficient for the stability of non linear problems \cite{Chavent89}.}. This condition is obviously fulfilled in case the solution blows up. We also proved that the numerical solution blows up in a finite time $T(h)$, and that $T(h)$ converges toward the theoretical blow-up time as $h$ gets smaller. We illustrate the performance of our method throughout several numerical tests and benchmarks.
%
%%%%%%%%%%%%%%%%%%%%%%%%%%%%%%%%%%%%%%%%%%%%%%%%%%%%%%%%%%%%%%%%%%%%%%%%%%%%%%%%%%%%%%%%%%%%%
%
\appendix
\section{Matrices properties}\label{app_mat}
Since $(\varphi_j^i)_{1\leq j\leq k+1}$ is a Lagrange polynomial basis of $\IP_k[K_i]$, then for any $x\in K_i$ we have
$$\dis{\sum_{j=1}^{k+1}\varphi_j^i(x) = 1 \ \text{ and }\ \sum_{j=1}^{k+1}(\varphi_j^i)'(x) = 0}.$$
On the other hand, using the transform $\varphi_j^i = \varphi_j \circ (\gamma^i)^{-1}$ where $\varphi_j$ is the $j$\up{th} Lagrange polynomial over $[-1,1]$ and $\gamma^i : [-1,1] \rightarrow K_i$, $x \mapsto \frac{1}{2}(h_i\,x + x_{i+\frac{1}{2}} + x_{i-\frac{1}{2}})$, one can easily show $M^i = h_i M$, $R^i = R$, $A^i=A$, $B^i=B$, $C^i=C$ and $D^i=D$ for all $i$ with
\begin{equation*}
M_{j\ell} = \dfrac{1}{2}\int_{-1}^1 \varphi_j\ \varphi_\ell\ dx,\quad R_{j\ell} = \int_{-1}^1  \varphi_j\ \varphi'_\ell\ dx,
\end{equation*}
\begin{equation*}
A_{j\ell} = \varphi_j(-1)\ \varphi_\ell(-1), \quad B_{j\ell} = \varphi_j(-1)\ \varphi_\ell(1),
\end{equation*}
and
\begin{equation*}
C_{j\ell} = \varphi_j(1)\ \varphi_\ell(-1), \quad D_{j\ell} = \varphi_j(1)\ \varphi_\ell(1).
\end{equation*}
It follows $\forall\ 1\leq j \leq k+1$,
\begin{equation*}
 \sum_{\ell=1}^{k+1}R_{j\ell} = \sum_{\ell=1}^{k+1} \int_{-1}^{1}\varphi_j(x) \varphi_\ell'(x) dx = \int_{-1}^{1}\varphi_j(x) \left(\sum_{\ell=1}^{k+1} \varphi_\ell'(x)\right) dx = 0,
\end{equation*}
\begin{equation*}
\sum_{\ell=1}^{k+1} A_{j\ell} = \sum_{\ell=1}^{k+1} \varphi_j(-1)\varphi_\ell(-1) = \varphi_j(-1) \sum_{\ell=1}^{k+1} \varphi_\ell(-1) = \varphi_j(-1),
\end{equation*}
\begin{equation*}
\sum_{\ell=1}^{k+1} B_{j\ell} = \sum_{\ell=1}^{k+1} \varphi_j(-1)\varphi_\ell(1) = \varphi_j(-1) \sum_{\ell=1}^{k+1} \varphi_\ell(1) = \varphi_j(-1).
\end{equation*}
Therefore, $\forall\ 1\leq j \leq k+1$,
\begin{equation}\label{sum_jl}
\sum_{\ell=1}^{k+1} (R_{j\ell} + A_{j\ell} - B_{j\ell}) = 0.
\end{equation}
Now, we have $$E = h_i(M^i)^{-1}(R^i+A^i) = M^{-1}(R+A) \quad \text{and }  \ F=-h_i(M^i)^{-1}B^i = -M^{-1}B$$ are constant matrices, and $\forall\ 1\leq j \leq k+1$
\begin{align}\label{som_E+F}
\sum_{\ell=1}^{k+1} E_{j\ell} + F_{j\ell} & = \sum_{\ell=1}^{k+1} M^{-1}(R+A-B)_{j\ell} \nonumber \\
& = \sum_{\ell=1}^{k+1} \sum_{s=1}^{k+1} (M^{-1})_{js}(R+A-B)_{s\ell}\nonumber \\
& = \sum_{s=1}^{k+1}(M^{-1})_{js} \sum_{\ell=1}^{k+1}(R_{s\ell}+A_{s\ell}-B_{s\ell})\nonumber \\
& = 0
\end{align}
where the last equality follows from \eqref{sum_jl}.
%
%%%%%%%%%%%%%%%%%%%%%%%%%%%%%%%%%%%%%%%%%%%%%%%%%%%%%%%%%%%%%%%%%%%%%%%%%%%%%%%%%%%%%%%%%%%%%
%
\section{Proof of Lemma \ref{lemma1}}\label{app_proof_lem}
We rewrite $M_n$ as $M_n = I - \widetilde{M}_n$ with
\begin{equation*}
\widetilde{M}_n = \begin{pmatrix}
\tilde{\mathcal{M}}_A^1 & 0 & \dots & 0 & \mathcal{M}_B^1 \\
\mathcal{M}_B^2 & \tilde{\mathcal{M}}_A^2 & 0 & \dots & 0 \\
0 & \mathcal{M}_B^3 & \tilde{\mathcal{M}}_A^3 & \ddots & \vdots\\
\vdots & \ddots & \ddots & \ddots & 0 \\
0 & \dots & 0 & \mathcal{M}_B^I & \tilde{\mathcal{M}}_A^I
\end{pmatrix}
\end{equation*}
with $$\tilde{\mathcal{M}}_A^i = \Delta t^n (M^i)^{-1}(R^i+A^i) = \dfrac{\Delta t^n}{h_i}E$$ and $$\mathcal{M}_B^i = -\Delta t^n (M^i)^{-1}B^i = \dfrac{\Delta t^n}{h_i}F.$$
Let $\left[\widetilde{M}_n\right]^i$ be the $i$\up{th} block-row of $\widetilde{M}_n$ and denote $x^+ = \max(x,0)$ and $x^- = \min(x,0)$ for any $x\in \IR$. Then, using \eqref{som_E+F}, we obtain for any $1\leq j \leq k+1$
\begin{align*}
\sum_{\ell=1}^{k+1}\left|\left[\widetilde{M}_n\right]^i_{j\ell}\right| & = \dfrac{\Delta t^n}{h_i} \sum_{\ell=1}^{k+1} \left|E_{j\ell} + F_{j\ell}\right| \\
& = \dfrac{\Delta t^n}{h_i} \left(\sum_{\ell=1}^{k+1} \left(E_{j\ell}^+ + F_{j\ell}^+\right) - \sum_{\ell=1}^{k+1} \left(E_{j\ell}^- + F_{j\ell}^-\right)\right) \\
& = 2\dfrac{\Delta t^n}{h_i} \sum_{\ell=1}^{k+1} \left(E_{j\ell}^+ + F_{j\ell}^+\right).
\end{align*}
It follows
\begin{align*}
\|\widetilde{M}_n\|_\infty & = \max_{1\leq i \leq I}\|\left[\widetilde{M}_n\right]^i\|_\infty \\
& = \max_{1\leq i \leq I}\left(\max_{1\leq j \leq k+1}\sum_{\ell=1}^{k+1}\left|\left[\widetilde{M}_n\right]^i_{j\ell}\right|\right)\\
& = \max_{1\leq i \leq I}\max_{1\leq j \leq k+1} 2\dfrac{\Delta t^n}{h_i} \sum_{\ell=1}^{k+1} \left(E_{j\ell}^+ + F_{j\ell}^+\right).
\end{align*}
In particular, if $h_i=h$ for all $i$, and if we denote $\rho = \sum_{\ell=1}^{k+1} \left(E_{j\ell}^+ + F_{j\ell}^+\right)$, then we obtain
\begin{align*}
\|M_n\|_\infty \leq \|I\|_\infty + \|\widetilde{M}_n\|_\infty = 1+2\rho\dfrac{\Delta t^n}{h}
\end{align*}
The same reasoning can be applied to the matrix $N_n$.
%
%%%%%%%%%%%%%%%%%%%%%%%%%%%%%%%%%%%%%%%%%%%%%%%%%%%%%%%%%%%%%%%%%%%%%%%%%%%%%%%%%%%%%%%%%%%%%
%
\section{Proof of Lemma \ref{lem_Kh>Khp}}\label{app_lem_lambda}
Let $1\leq i \leq I$, then using the classical inequality $\left|\sum_{j=1}^m a_j\right|^p \leq m^{p-1} \sum_{j=1}^m |a_j|^p$ we obtain
\begin{align*}
\left|\int_{K_i} u^{i,n+1}_h(x)\,dx\right|^p & = \left|\int_{K_i} \sum_{j=1}^{k+1} u_j^{i,n+1}\varphi^i_j(x)\,dx\right|^p \\
& \leq (k+1)^{p-1}\sum_{j=1}^{k+1} \left|u_j^{i,n+1}\right|^p \left|\int_{K_i}\varphi^i_j(x)\,dx\right|^p \\
& = (k+1)^{p-1}\sum_{j=1}^{k+1} \left|u_j^{i,n+1}\right|^p \left|\dfrac{h}{2}\int_{-1}^1\varphi_j(x)\,dx\right|^p.
\end{align*}
Denote $$\lambda = \left(\frac{k+1}{2}\max_{1\leq j \leq k+1}\left|\int_{-1}^1\varphi_j(x)\,dx\right|\right)^{1-p},$$ then we have
\begin{align*}
\left(K_h(u^{n+1}_h)\right)^p & = \dfrac{1}{(b-a)^p}\left(\sum_{i=1}^I\int_{K_i} u^{i,n+1}_h(x)\,dx\right)^p \\
& \leq \dfrac{I^{p-1}}{(b-a)^p}\sum_{i=1}^I\left|\int_{K_i} u^{i,n+1}_h(x)\,dx\right|^p\\
& \leq \dfrac{(Ih)^{p-1}}{\lambda (b-a)^p} \sum_{i=1}^I \sum_{j=1}^{k+1} \left|u_j^{i,n+1}\right|^p \left|\dfrac{h}{2}\int_{-1}^1\varphi_j(x)\,dx\right| \\
% & \leq \dfrac{1}{(\lambda(b-a)} \sum_{j=1}^{k+1} \left|u_j^{i,n+1}\right|^p \left|\dfrac{h}{2}\int_{-1}^1\varphi_j(x)\,dx\right|. %\\
% & = \dfrac{1}{\lambda} \sum_{j=1}^{k+1} \left|u_j^{i,n+1}\right|^p \left|\int_{K_i}\varphi^i_j(x)\,dx\right|.
\end{align*}
Now, if $0\leq k\leq 7$ then the integrals $\int_{-1}^1\varphi_j(x)\,dx $ are positives for all $1\leq j \leq k+1$ (see table \ref{tab_poly}). It follows
\begin{align*}
\left(K_h(u^{n+1}_h)\right)^p & \leq \dfrac{1}{\lambda} \sum_{j=1}^{k+1} \left|u_j^{i,n+1}\right|^p \int_{K_i}\varphi^i_j(x)\,dx \\
& = \dfrac{1}{\lambda} K_h\big{(}\p(|u^{n+1}_h|^p)\big{)}.
\end{align*}
\begin{table}
\begin{center}
\begin{tabular}{|c|c|c|c|c|c|c|c|c|}
\hline 
\backslashbox{$k$}{$\alpha_{j}$} & $\alpha_{1}$ & $\alpha_{2}$ & $\alpha_{3}$ & $\alpha_{4}$ & $\alpha_{5}$ & $\alpha_{6}$ & $\alpha_{7}$ & $\alpha_{8}$ \\ % & $\alpha_{9}$ \\ 
\hline 
$0$ & 2 &   &  &  &  &  &  &  \\ 
\hline 
$1$ & 1 & 1 &  &  &  &  &  &  \\ 
\hline 
$2$ & $\frac{1}{3}$ & $\frac{4}{3}$ & $\frac{1}{3}$ &  &  &  &  &  \\ 
\hline 
$3$ & $\frac{1}{4}$ & $\frac{3}{4}$ & $\frac{3}{4}$ & $\frac{1}{4}$ &  &  &  &  \\ 
\hline 
$4$ & $\frac{7}{45}$ & $\frac{32}{45}$ & $\frac{12}{45}$ & $\frac{32}{45}$ & $\frac{7}{45}$ &  &  &  \\ 
\hline 
$5$ & $\frac{19}{144}$ & $\frac{75}{144}$ & $\frac{50}{144}$ & $\frac{50}{144}$ & $\frac{75}{144}$ & $\frac{19}{144}$ &  &  \\ 
\hline 
$6$ & $\frac{41}{420}$ & $\frac{216}{420}$ & $\frac{27}{420}$ & $\frac{272}{420}$ & $\frac{27}{420}$ & $\frac{216}{420}$ & $\frac{41}{420}$ &  \\ 
\hline 
$7$ & $\frac{751}{8640}$ & $\frac{3577}{8640}$ & $\frac{1323}{8640}$ & $\frac{2989}{8640}$ & $\frac{2989}{8640}$ & $\frac{1323}{8640}$ & $\frac{3577}{8640}$ & $\frac{751}{8640}$ \\ 
\hline 
%$8$ & $\frac{989}{14175}$ & $\frac{5888}{14175}$ & $-\frac{928}{14175}$ & $\frac{10496}{14175}$ & $-\frac{4540}{14175}$ & $\frac{10496}{14175}$ & $-\frac{928}{14175}$ & $\frac{5888}{14175}$ & $\frac{989}{14175}$ \\
%\hline 
\end{tabular}
\caption{\label{tab_poly}Values of $\alpha_{j} := \int_{-1}^1 \varphi_{j}(x)dx$ where $\varphi_{j}$ is the $j$\up{th} Lagrange polynomial of degree $k$ over $[-1,1]$.}
\end{center}
\end{table}

\begin{thebibliography}{}

\bibitem{Abia} Abia, L.M., L\'opez-Marcos, J.C. and Mart\'inez, J. On the blow-up time convergence of semidiscretizations of reaction-diffusion equations, Appl. Numer. Math. {\bf 26}(4): 399-414, 1998.
          
\bibitem{Ab01} Abia, L.M., L\'opez-Marcos, J.C. and Mart\'inez, The Euler method in the numerical integration of reaction-diffusion problems with blow-up, Appl. Numer. Math. {\bf 38}: 287-313, 2001.

\bibitem{AM} Antonini, C. and Merle, F. Optimal bounds on positive blow-up solutions for a semilinear wave equation. Int. Math. Res. Notices, {\bf 21}: 1141-1167, 2001.

\bibitem{Az} Azaiez, A., Masmoudi, N. and Zaag, H. Blow-up rate for a semilinear wave equation with exponential nonlinearity in one space dimension, Math. Soc. Lect. Note Ser. {\bf 450}: 1-32, 2019.
            
\bibitem{Berg} Berger, M. and Kohn, R.V. A rescaling algorithm for the numerical calculation of blowing-up solutions. Com. Pure. Appl. Math., {\bf 41}(6): 841-863, 1988.

\bibitem{Bra04} Br\"andle, C., Groisman, P. and Rossi, J.D. Fully discrete adaptive methods for a blow-up problem. Math. Models. Meth. Appl. Sci., {\bf 14}(10): 1425-1450, 2004.

\bibitem{CF85} Caffarelli, L.A. and Friedman, A. Differentiability of the Blow-up Curve for one Dimensional Nonlinear Wave Equations. Arch. Rational Mech. Anal., {\bf 91}(1):83-98,1985.

\bibitem{Caffarelli and Friedman} Caffarelli, L.A. and Friedman, A. The blow-up boundary for nonlinear wave equations. Trans. Am. Math. Soc. {\bf 297}(1): 223-241, 1986.

\bibitem{Chavent89} Chavent, G. and Cockburn, B. The local projection $P0-P1$-discontinuous-Galerkin finite element method for scalar conservation laws. Mod. Math. Anal. Num. {\bf 23}(4): 565-592, 1989.

\bibitem{Chen} Chen, Y.G. Asymptotic behaviours of blowing-up solutions for finite difference analogue of $u_t = u_{xx} + u^{1+\alpha}$. J. Fac. Sci. Univ. Tokyo {\bf 33}: 541-574, 1986.

\bibitem{Cho et al} Cho, C.H., Hamada, S. and Okamoto, H. On the finite difference approximation for a parabolic blow-up problems, Japan J. Indust. Appl. Math. {\bf 24}: 105-134, 2007.

\bibitem{Cho} Cho, C.H. A finite difference scheme for blow-up solutions of nonlinear wave equations, Numer. Math. Theor. Meth. Appl. {\bf 3}(4): 475-498, 2010.

\bibitem{Cho13} Cho, C.H. On the computation of the numerical blow-up time, Japan J. Indust. Appl. Math. {\bf 30}: 331-349, 2013.

\bibitem{Cho16} Cho, C.H. Numerical detection of blow-up: a new sufficient
condition for blow-up, Japan J. Indust. Appl. Math. {\bf 33}(1): 81-98, 2016.

\bibitem{Cho17} Cho, C.H. On the computation for blow-up solutions of the nonlinear wave equation, Numer. Math. {\bf 138}: 537-556, 2018.

\bibitem{Ciarlet} Ciarlet, P. The Finite Element Method for Elliptic Problem, North Holland, 1975.

\bibitem{Cockburn} Cockburn, B. Discontinuous Galerkin methods for convection-dominated problems, in High-Order Methods for Computational Physics, T.J. Barth and H. Deconinck, editors, Lecture Notes in Computational Science and Engineering, Springer, {\bf 9}: 69-224, 1999.

\bibitem{Cockburn et Shu} Cockburn, B., Karniadakis, G.E. and Shu, C.-W. Discontinuous Galerkin methods: Theory, Computation and Application, Springer-Verlag, Berlin Heidelberg, 2000.

\bibitem{Cock_shu_89} Cockburn, and Shu, C.-W. TVB Runge-Kutta local projection discontinuous Galerkin finite element method for conservation laws II: general framework, Math. Comp., {\bf 52}: 411-435, 1989.

\bibitem{Cons85} Constantin, P., Lax, P.D. and Majda, A.J. A simple one-dimensional model for the three dimensional vorticity equation, Comm. Pure Appl. Math., {\bf 38}: 715-724, 1985.

\bibitem{CZ12} C{\^o}te, R. and Zaag, H. Construction of a multisoliton blowup solution to the semilinear wave equation in one space dimension, Comm. Pure Appl. Math., {\bf 66}(10): 1541-1581, 2013.

\bibitem{Ern} Ern, A. and Guermond, J.L. Theory and Practice of Finite Elements, Springer-Verlag, New York, 2004.

\bibitem{Evans} Evans, L.C. Partial Differential Equations, Graduate Studies in Math. {\bf 19}. American Mathematical Society, Providence ,1998.

\bibitem{Gla73} Glassey, R.T. Blow-up Theorems for nonlinear wave equations. Math. Z. {\bf 132}: 183-203, 1973.

\bibitem{Glassey} Glassey, R.T. Finite-time blow-up for solutions of nonlinear wave equations. Math. Z. {\bf 177}: 323-340, 1981.

\bibitem{Groisman06} Groisman, P. Totally Discrete Explicit and Semi-implicit Euler Methods for a Blow-up Problem in Several Space Dimensions. Computing, {\bf 76}: 325-352, 2006.

\bibitem{Grot} Grote, M.J., Schneebeli, A. and Sch\"otzau, D. Discontinuous Galerkin Finite Element Method for the Wave Equation, SIAM. J. Numer. Anal. {\bf 44}(6): 2408-2431, 2006.

\bibitem{Guo15} Guo, L. and Yang, Y. Positivity preserving high-order local discontinuous Galerkin method for parabolic equations with blow-up solutions, J. Comput. Phys. {\bf 289}: 181-195, 2015.

\bibitem{Hest} Hesthaven, J. and Warburton, T. Nodal Discontinuous Galerkin Methods : Algorithms, Analysis, and Applications. Texts in Applied Mathematics, Springer, 2008.
            
\bibitem{Holm18} Holm, B. and Wihler, T.P. Continuous and discontinuous Galerkin time stepping methods for nonlinear initial value problems with application to finite time blow-up, Numer. Math. {\bf 138}(3): 767-799, 2018.

\bibitem{Hu} Hu, Q.F., Hussaini, M.Y. and Rasetarinera, P. An Analysis of the Discontinuous Galerkin Method for Wave Propagation Problems, J. Comput. Phys. {\bf 151}(2): 921-946, 1999.

\bibitem{John} John, F. Blow-up of solutions of nonlinear wave equations in three space dimensions. Manuscripta Math., {\bf 28}: 235-268, 1979.

\bibitem{livine} Levine, H.A. Instability and nonexistence of global solutions to nonlinear wave equations of the form $Pu_{tt} = -Au + F(u)$. Trans. Amer. Math. Soc., {\bf 192}: 1-21, 1974.

\bibitem{Matsuo} Matsuo, T. New conservative schemes with discrete variational derivatives for nonlinear wave equations. J. Comput. Appl. Math., {\bf 203}: 32-56, 2007.

\bibitem{MZ} Merle, F. and Zaag, H. Determination of the blow-up rate for the semilinear wave equation. Amer. J. Math., {\bf 125}(5): 1147-1164, 2003.

\bibitem{MR2147056} Merle, F. and Zaag, H. On growth rate near the blowup surface for semilinear wave equations. Int. Math. Res. Not., {\bf 19}: 1127-1155, 2005.

\bibitem{MR2362418} Merle, F. and Zaag, H. Existence and universality of the blow-up profile for the semilinear wave equation in one space dimension. J. Funct. Anal., {\bf 253}(1): 43-121, 2007.

\bibitem{MZ12} Merle, F. and Zaag, H. Existence and classification of characteristic points at blow-up for a semilinear wave equation in one space dimension. Amer. J. Math., {\bf 134}(3): 581-648, 2012.

\bibitem{Nakagawa} Nakagawa, T. Blowing up of a finite difference solution to $u_t = u_{xx} + u^2$. Appl. Math. Optim., {\bf 2}: 337-350, 1976.
             
\bibitem{Ngu17} Nguyen, V.T. Numerical analysis of the rescaling method for parabolic problems with blow-up in finite time. Physica D, {\bf 339}: 49-65, 2017.

% \bibitem{Piperno} Piperno, S. DGTD methods using modal basis functions and symplectic local time-stepping: Application to wave propagation problems. Euro. J. Comput. Mech. {\bf 15}(6): 643-670, 2006.

\bibitem{Sasaki} Saito, N. and Sasaki, T. Blow-up of finite-difference solutions to nonlinear wave equations. J. Math. Sci. Univ. Tokyo {\bf 23}(1): 349-380, 2016.
\end{thebibliography}
\end{document}